\documentclass[10pt]{amsart}
\keywords{K- and L-Theory, C*-Algebras, Homotopy theory}
\subjclass[2010]{19K35, 46L80, 55P42}
\usepackage{a4wide}
\usepackage[utf8]{inputenc}
\usepackage{amsmath, amssymb, bbm}
\usepackage{amsthm}
\usepackage{mathtools}
\usepackage{bbm}
\usepackage{blkarray}
\usepackage[matrix,arrow,curve]{xy}
\usepackage{tikz}
\usepackage{graphicx}
\usepackage{color}
\usepackage[normalem]{ulem}
\usepackage{enumerate}
\definecolor{darkblue}{rgb}{0,0,0.6}
\usepackage[ocgcolorlinks,colorlinks=true, citecolor=darkblue, filecolor=darkblue, linkcolor=darkblue, urlcolor=darkblue]{hyperref}
\usepackage[capitalize,noabbrev]{cleveref}
\usepackage{comment}
\usepackage[mathscr]{euscript}

\newcounter{commentcounter}

\protected\def\ignorethis#1\endignorethis{}
\let\endignorethis\relax
\newtheoremstyle{thms}
	{}{}{\itshape}{}{\bfseries }{}{ }
	{\thmname{#1} \thmnumber{#2}. \thmnote{\bfseries{[#3]}}}
\newtheoremstyle{thms2}
	{}{}{\itshape}{}{\bfseries }{}{ }
	{\thmname{#1}\thmnumber{#2}. \thmnote{\bfseries{[#3]}}}

\newtheoremstyle{name}
	{}{}{\itshape}{}{\bfseries }{}{ }
	{\thmname{#1}\thmnumber{#2}\thmnote{\bfseries{[#3]}}}
\newtheoremstyle{defs}
	{}{12pt}{\normalfont}{}{\bfseries }{}{ }
	{\thmname{#1} \thmnumber{#2}. \thmnote{\bfseries{(#3)}}}
\newtheoremstyle{defs2}
	{}{12pt}{\normalfont}{}{\bfseries }{}{ }
	{\thmname{#1}\thmnumber{#2}. \thmnote{\bfseries{(#3)}}}
\newtheoremstyle{rmk}
	{}{}{\normalfont}{}{\itshape }{}{ }{\thmname{#1}. \thmnote{#3}}
\newtheoremstyle{claim}
	{}{}{\normalfont}{}{\itshape}{}{ }{\thmname{#1} \thmnumber{#2}. \thmnote{#3}}

\theoremstyle{thms}
\newtheorem{Prop}{Proposition}[section]
\newtheorem{Thm}[Prop]{Theorem}%[section]
\newtheorem{Lemma}[Prop]{Lemma}%[section]
\newtheorem{Cor}[Prop]{Corollary}%[section]

\theoremstyle{thms2}
\newtheorem*{ThmA}{Theorem A}
\newtheorem*{ThmB}{Theorem B}
\newtheorem*{ThmC}{Theorem C}
\newtheorem*{ThmD}{Theorem D}
\newtheorem*{ThmE}{Theorem E}

\newtheorem*{Conj}{Conjecture}%[section]

\theoremstyle{name}

\theoremstyle{defs}
\newtheorem{Def}[Prop]{Definition}%[section]

\theoremstyle{defs2}
\newtheorem*{ackn}{Acknowledgements}

\theoremstyle{rmk}
\newtheorem*{Rmk}{Remark}

\theoremstyle{rmk}

\theoremstyle{rmk}

\theoremstyle{claim}

\newcommand{\Z}{\mathbb{Z}}
\newcommand{\F}{\mathbb{F}}
\newcommand{\Q}{\mathbb{Q}}
\newcommand{\C}{\mathbb{C}}
\newcommand{\R}{\mathbb{R}}
\newcommand{\cc}{\mathrm{C^*Alg}}

\newcommand{\rc}{\mathrm{R^*Alg}}
\newcommand{\ccu}{\cc^{\mathrm{unit}}}

\newcommand{\Sp}{\mathrm{Sp}}

\renewcommand{\a}{\mathrm{alg}}

\newcommand{\Ab}{\mathrm{Ab}}

\newcommand{\ringsinv}{\mathrm{Ring}_{\mathrm{inv}}}
\newcommand{\KK}{\mathrm{KK}}
\newcommand{\Fun}{\mathrm{Fun}}

\newcommand{\HZ}{\mathrm{H}\Z}

\newcommand{\MSO}{\mathrm{MSO}}
\newcommand{\id}{\mathrm{id}}

\newcommand{\N}{N}

\newcommand{\s}{\mathscr{S}}
\newcommand{\Grp}{\mathrm{Grp}}
\newcommand{\e}{\mathbb{E}}

\newcommand{\fib}{\mathrm{fib}}

\newcommand{\wt}{\widetilde}
\newcommand{\adj}{[\tfrac{1}{2}]}
\newcommand{\map}{\mathrm{map}}
\newcommand{\Map}{\mathrm{Map}}
\newcommand{\Hom}{\mathrm{Hom}}
\newcommand{\Ext}{\mathrm{Ext}}
\newcommand{\Cat}{\mathrm{Cat}}
\newcommand{\lex}{\mathrm{Lex}}
\newcommand{\lax}{\mathrm{lax}}
\newcommand{\Alg}{\mathrm{Alg}}
\renewcommand{\c}{\mathscr{C}}
\newcommand{\Nfin}{\mathrm{NFin}_*}

\newcommand{\Ob}{\mathrm{Ob}}
\newcommand{\Set}{\mathrm{Set}}
\renewcommand{\b}{\mathcal{B}}
\newcommand{\h}{\mathcal{H}}
\renewcommand{\N}{\mathrm{N}}
\newcommand{\coker}{\mathrm{coker}}

\renewcommand{\S}{\mathbb{S}}

\newcommand{\Gpd}{\mathrm{Gpd}}
\newcommand{\BC}{\mathrm{BC}}
\newcommand{\FJ}{\mathrm{FJ}}

\newcommand{\Dd}{\mathscr{D}}
\newcommand{\Cc}{\mathscr{C}}
\newcommand{\MU}{\mathrm{MU}}
\newcommand{\BU}{\mathrm{BU}}

\newcommand{\calK}{\mathcal{K}}
\renewcommand{\emptyset}{\varnothing}

\title{On the Relation between K- and L-Theory of $C^*$-Algebras}
\author[M.~Land]{Markus Land}
\address{Rheinische Friedrich-Wilhelms Universit\"at Bonn, Mathematisches Institut, Endenicher Allee 60, 53115 Bonn, Germany}
\email{land@math.uni-bonn.de}

\author[T.~Nikolaus]{Thomas Nikolaus}
\address{Max Planck Institut f\"ur Mathematik, Vivatsgasse 7, 53111 Bonn, Germany}
\email{nikolaus@math.uni-bonn.de}

\date{\today}

\begin{document}

\begin{abstract}
We prove the existence of a map of spectra $\tau_A\colon kA \to \ell A$ between connective topological $K$-theory and connective algebraic $L$-theory of a complex $C^*$-algebra $A$ which is natural in $A$ and compatible with multiplicative structures. We determine its effect on homotopy groups and as a consequence obtain a natural equivalence $KA\adj \xrightarrow{\simeq} LA\adj$ of periodic $K$- and $L$-theory spectra after inverting $2$. We show that this equivalence extends to $K$- and $L$-theory of real $C^*$-algebras. Using this we give a comparison between the real Baum-Connes conjecture and the $L$-theoretic Farrell-Jones conjecture. We conclude that these conjectures are equivalent after inverting $2$ if and only if a certain completion conjecture in $L$-theory is true.
\end{abstract}

\maketitle
\tableofcontents

\begin{section}{Introduction}

\subsection{Motivation}
One of the main motivations of this article is to give a precise comparison between the Baum-Connes conjecture in real topological $K$-theory and the  Farrell-Jones conjecture in algebraic $L$-theory. We briefly recall the setup of these conjectures. For this, let $G$ be a countable discrete group. We then consider the integral group ring $\Z G$ with its canonical involution induced by $g \mapsto g^{-1}$. Similarly we can consider the reduced group $C^*$-algebra $C^*_r G$ which is a specific completion of the complex group ring $\C G$ or its real counterpart $C^*_r(G;\R)$ which is a completion of $\R G$. The Baum-Connes conjecture then predicts that a certain assembly map
\[\xymatrix{KO_*^G(\underline{E}G) \ar[r]^-\BC & KO_*(C^*_r(G;\R))}\]
is an isomorphism. Here, the left hand side denotes the $G$-equivariant $KO$-homology of the classifying space for finite subgroups of $G$. Similarly the $L$-theoretic Farrell-Jones conjecture predicts that the assembly map
\[\xymatrix{(L^q\Z)^G_*(\uuline{E}G) \ar[r]^-\FJ & L^q_*(\Z G)}\]
is an isomorphism. Here the left hand side is the equivariant homology theory associated to quadratic $L$-theory of the integers evaluated on the classifying space for virtually cyclic subgroups. We suppress the decoration in $L$-theory, but implicitly mean the decoration $\langle -\infty \rangle$ in the introduction. See  \cite{DavisLueck} for details.

Both these conjectures have been studied extensively due to their intimate relation to other prominent conjectures, such as the Borel conjecture, the Novikov conjecture, the Kadison conjecture and the stable Gromov-Lawson-Rosenberg conjecture. For details about the construction, status and relation to these other conjectures we recommend the survey \cite{LueckReich}. 
We recall from \cite{DavisLueck} that these assembly maps are constructed using the fact that the $K$- resp. the $L$-groups are homotopy groups of $K$- resp. $L$-theory spectra. 
\\

When trying to relate these two conjectures a first step is to relate the codomains of the assembly maps, i.e the two groups $KO_*(C^*_r(G;\R))$ and $L^q_*(\Z G)$. The canonical inclusion $\Z G \subseteq C^*_r(G;\R)$ induces a map in quadratic $L$-theory, so that one is naturally led to try to relate the real $K$-theory group of the group $C^*$-algebra to its $L$-theory. Indeed, the relation between $K$- and $L$-groups of $C^*$-algebras  in general has been greatly studied, see e.g. \cite{Miller,Karoubi,Rosenberg}: for complex $C^*$-algebras these two groups are  naturally isomorphic. Over the real numbers the precise relation is more complicated, but it turns out that after inverting $2$, the $KO$- and $L$-groups are still isomorphic, see for instance \cite[Theorem 1.11]{Rosenberg} or Theorem C in \cref{ThmC}. 

Unfortunately, this is not enough to deduce that the assembly maps in question are isomorphic. Indeed by the construction of the assembly maps as in \cite{DavisLueck}, one needs to understand the relation between the two spectrum valued functors $KO$ and $L$ and not just their homotopy groups.

Our main contribution to this is that we show that the two functors $KO$ and $L$ become equivalent after inverting 2. In particular we obtain a commutative diagram relating the BC-assembly map and the FJ-assembly map after inverting 2. In fact we also study the complex case and even integral results before inverting 2. We will give precise statements of our results in \cref{ThmD}. But let us first comment on the relation to previous work.\\

Our strategy of relating the assembly maps is taken from \cite[Proposition 3.19]{LueckReich} and we thank Wolfgang L\"uck for sharing his ideas with us.  For the key input, namely the 
 equivalence of spectrum valued $K$- and $L$-theory functors after inverting 2, the authors refer to a result of Rosenberg, \cite[Theorem 1.11 and Theorem 2.1]{Rosenberg}. Indeed, in \cite[Theorem 2.7]{Rosenberg} Rosenberg also considers a diagram relating the Farrell-Jones conjecture to the Baum-Connes conjecture. See also \cite[Page 12]{Rosenberg2} for a similar statement.
 
 However we believe that Rosenberg's proof of the natural equivalence $KO\adj \simeq L\adj$ is not sufficient. The starting point of his argument is to show that there is an equivalence of homotopy ring spectra $KO\adj \simeq L\R\adj$.
 This is essentially concluded from a known equivalence of underlying spaces due to Sullivan (see \cite[Theorem 4.28]{MM}) but we think the arguments given by Rosenberg are not enough to upgrade this to an equivalence of homotopy ring spectra. However it was known by work of Taylor and Williams \cite{Williams} that they are equivalent as spectra. Lurie gives a proof that they are equivalent as homotopy ring spectra in Lecture 25 of \cite{Lurie}.
 
Taking this for granted, Rosenberg then views both $LA\adj$ and $KO(A)\adj$ as module spectra over 
 $KO\adj$. Using the classification of such module spectra due to Bousfield \cite[Theorem 9.6]{Bousfield} one can lift the isomorphism of homotopy groups to an equivalence of spectra. It turns out that this lift is in general not unique, the ambiguity is given by an $\Ext$-group which will generically not vanish. Thus this gives an equivalence $LA\adj \simeq KO(A)\adj$ for every real $C^*$-algebra $A$ individually but there is no reason why one can find an equivalence $KO(A)\adj \simeq LA\adj$ which is natural in $A$, not even in the homotopy category of spectra, let alone in a more structured setting as needed for the comparison of assembly maps.

\subsection{Statement of results}

Throughout this  paper we will freely use the language of $\infty$-categories as developed by Joyal \cite{Joyal} and Lurie \cite{LurieHTT, LurieHA}. We follow Lurie's notational conventions. 
Essentially all of the results are also expressible in other frameworks of abstract homotopy theory, e.g. Quillen model categories. 
\newline

We study the space of maps between $K$ and $L$ viewed as functors between the $\infty$-category $\N\cc$ associated to the 1-category $\cc$ of  complex, separable $C^*$-algebras and  the $\infty$-category $\Sp$ of spectra. See Section \ref{prelims} and the appendix for details about these functors, in particular the fact that they admit lax symmetric monoidal structures.  For the first statement we consider their connective covers, by which mean the functors $k,\ell: \N\cc \to \Sp_{\geq 0}$ obtained by post composing $K$ and $L$ with the connective cover functor $\tau_{\geq 0}: \Sp \to \Sp_{\geq 0}$. 

\begin{ThmA}\label{ThmA}
For every $n \in \Z$ there exists a natural transformation $\tau(n) \colon k \to \ell$, unique up to homotopy, characterized by the property that $\tau(n)_\C \colon \pi_0(ku) \to \pi_0(\ell\C)$ is given by multiplication with $n$. More precisely the map
\[ \xymatrix@R=.3cm{\pi_0\big( \Map_{\Fun(\N\cc,\Sp_{\geq 0})}(k,\ell) \big) \ar[r] & \Z \\ [\eta] \ar@{|->}[r] & \pi_0(\eta_\C)(1)}\]
is a bijection. Moreover, there is an essentially unique  lax symmetric monoidal transformation $\tau \colon k \to \ell$. Its underlying transformation is $\tau(1)$. 
\end{ThmA}

We continue by calculating the effect on homotopy of the multiplicative transformation $\tau$ just considered.

\begin{ThmB}\label{ThmB}
For $i \in \{0,1\}$, all $k \geq 0$, and all $A \in \cc$ there is an exact sequence
\[\xymatrix{0 \ar[r] & \pi_{2k+i}(kA)_{2^k} \ar[r] & \pi_{2k+i}(kA) \ar[r]^-{\tau_A} & \pi_{2k+i}(\ell A) \ar[r] & \frac{\pi_{2k+i}(\ell A)}{2^k\cdot\pi_{2k+i}(\ell A)} \ar[r] & 0}.\]
\end{ThmB}

As a consequence one can also deduce results about the periodic versions of these functors if one inverts 2. For the real statement we consider $KO$ and $L$ as functors from the $\infty$-category $\N\rc$ associated to the 1-category $\rc$ of real $C^*$-algebras. 
\begin{ThmC}\label{ThmC}
The functors $K\adj, L\adj\colon \N\cc \to \Sp$ are equivalent as lax symmetric monoidal functors. 
Also the two functors $KO\adj, L\adj\colon \N\rc \to \Sp$ are equivalent as lax symmetric monoidal functors.
\end{ThmC}

So far we are not able to produce an integral transformation for real $C^*$-algebras as in the complex case (Theorem A). 
But we can still use the equivalence after inverting 2 to relate the $L$-theoretic Farrell-Jones conjecture to the real Baum-Connes conjecture. The following theorem makes this precise.

\begin{ThmD}\label{ThmD}
Let $G$ be a countable discrete group. Then there is a commutative diagram in which the horizontal arrows are the respective assembly maps appearing in  the Baum-Connes and the Farrell-Jones conjecture:
\[\xymatrix@C=1.5cm{ KO^G_*(\underline{E}G)\adj \ar[rr]^-{\BC\adj} \ar[d]_\cong & & KO_*(C^*_r (G;\R))\adj  \ar[d]^-\cong \\ L\R^G_*(\underline{E}G)\adj \ar[r]^-{\FJ\adj} & L_*(\R G)\adj \ar[r] & L_*(C^*_r (G;\R))\adj \\ L^q\Z^G_*(\underline{E}G)\adj \ar[u]^\cong \ar[r]_-{\FJ\adj} & L^q_*(\Z G)\adj \ar[u] &}\]
\end{ThmD}

We end with a theorem which implies that the integral map, which we construct in Theorem A, is the only non-trivial integral map between the variants of the functors $K$ and $L$ that is possible. Indeed we prove a stronger statement about the spectra $L\C$ and $KU = K\C$.

\begin{ThmE}\label{ThmE}
We have that
\[ [L\C,KU] = [KU,L\C]  = [\ell\C,ku] = 0 \]
where $[-,-]$ denotes the groups of morphisms in the homotopy category of spectra.
\end{ThmE}

\subsection{Organisation of the paper}
In \cref{prelims}
we deal with the necessary preliminaries. We first recall $K$-theory of $C^*$-algebras in \cref{K-theory} and then move on to $L$-theory of general rings in \cref{L-theory} and $L$-theory of $C^*$-algebras in \cref{L-theory of CAlg}.
 \cref{Construction of the map} is devoted to the proof of Theorem A.  In \cref{KK infty category} we introduce a stable $\infty$-category $\KK_\infty$ which is a key player for all constructions. We prove the first part of Theorem A in \cref{construction} and continue to discuss  multiplicative properties in \cref{Multiplicative properties} to finish the proof of Theorem A. \cref{effect on homotopy} is devoted to a proof of Theorem B. As preliminaries we need to discuss excision properties of $L$-theory, which we deal with in \cref{excision}. In \cref{Applications} we first prove Theorem C in \cref{real} and use this to establish Theorem D in \cref{assembly}. Finally in \cref{integral maps} we prove Theorem E and in \cref{non unital L-theory} we extend the lax symmetric monoidal structure of $L$-theory to non-unital $C^*$-algebras.

\end{section}
\begin{ackn}
We would like to thank Wolfang L\"uck for sharing with us his ideas and insights about $K$- and $L$-theory. We are pleased to thank Christian Wimmer for several fruitful discussions. The first author would like to thank the GRK 1150 - Cohomology and Homotopy for the support during his time as Phd student at the university of Bonn. Furthermore he was supported by Wolfgang L\"ucks ERC Advanced Grant ``KL2MG-interactions'' (no.662400) granted by the European Research Council.
\end{ackn}

\section{Preliminaries about $K$- and $L$-theory}\label{prelims}

In this section we first develop all necessary tools to construct the transformation between $K$-theory and $L$-theory. 
We start by recalling basic definitions and properties of $K$-theory and $L$-theory. 

%%%%%%%%%%%%%%%%%%%%%%%%%%%%%%%%%%%%%%%%%%%%%%%%%%%%%%%%%%%%%%

\subsection{$C^*$-algebras and $K$-theory}\label{K-theory}
Throughout this paper we work in the category of complex, \emph{separable}\footnote{This is needed in order to turn $C^*$-algebras with the Kasparov product into a category.}, not necessarily unital $C^*$-algebras with not necessarily unital $^*$-homomorphisms as morphisms. We denote this category by $\cc$.

To be more specific we recall that a $C^*$-algebra is a complex Banach algebra $A$ equipped with a complex antilinear involution $x \mapsto x^*$ that satisfies the $C^*$-identity
\[ \| x^*x \| = \|x\|^2 \text{ for all } x \in A.\]
Obviously there is a forgetful functor
\[\xymatrix{\ccu \ar[r]^-U & \ringsinv }\]
from the category of unital $C^*$-algebras and unital morphisms to the category of involutive rings, by forgetting the topology on $A$. It is well known that this functor is fully faithful when viewed as a functor to involutive $\C$-algebras, see \cite[Chapter I, section 5]{Takesaki}, and that one can reconstruct the norm on $A$ from the involutive ring $UA$, \cite[Chapter I, Prop. 4.2]{Takesaki}.

Examples of $C^*$-algebras are continuous functions (vanishing at infinity) of a (locally) compact Hausdorff space $X$, denoted by $C(X)$. Indeed, the theorem of Gelfand and Naimark states that any commutative $C^*$-algebra is of this kind, see \cite[Theorem 4.4.]{Takesaki}. Further examples are $\b(\h)$, the bounded operators on a Hilbert space, and thus also any norm-closed sub-$*$-algebra of $\b(\h)$. It is a theorem of Gelfand, Naimark, and Segal that any $C^*$-algebra admits a faithful representation on a Hilbert space and is thus a norm-closed subalgebra of $\b(\h)$ for some Hilbert space $\h$, see \cite[Theorem 9.18]{Takesaki}.

\begin{Lemma}\label{lemma unitalization}
The inclusion functor $\ccu \to \cc$ admits a left adjoint, called the \emph{unitalization}
\[ \xymatrix@R=.3cm{\cc \ar[r] & \ccu \\ A \ar@{|->}[r] & A^+ }\]
which comes with a natural split short exact sequence
\[\xymatrix{0 \ar[r] & A \ar[r] & A^+ \ar[r]^-{\pi_A} & \C \ar[r] & 0}.\]
If $A$ is unital, then $A^+ \cong A\times \C$.
\end{Lemma}
\begin{proof}
First one observes that if $A$ is unital then the algebra $A\times \C$ is unital as well and satisfies the universal property needed. Thus the furthermore part is easy. 
If $A$ does not have a unit one considers the embedding
\[ A \subseteq \b(A) \]
by left-multiplication. It is injective, hence isometric and the image does not contain the unit (by the assumption that $A$ does not have a unit element). The smallest subalgebra containing both $A$ and the identity of $\b(A)$ is the $C^*$-algebra $A^+$. Notice that as a $\C$-vector space $A^+$ is of the form $A\oplus \C$ but the multiplication is twisted:
\[ (a,\lambda)\cdot(b,\mu) = (ab+\lambda b+\mu a, \lambda\mu)\]
from which it follows that $A$ is an ideal in $A^+$ whose quotient is $\C$.
See also \cite[Proposition 1.5]{Takesaki}.
\end{proof}
We want to remark that this is the minimal way to embedd $A$ as an ideal in a unital $C^*$-algebra. In general there is a whole family of unitalizations one can consider. If $A$ is commutative, and thus of the form $C_0(X)$ for a locally compact Hausdorff space $X$ these unitalizations correspond precisely to compactifications of $X$. For example, the unitalization we discussed above corresponds to the one-point compactification.

Another important technical tool is the tensor product of $C^*$-algebras. Let $A$ and $B$ be $C^*$-algebras. Then we denote the maximal tensor product by $A\otimes B$. It is a $C^*$-completion of the algebraic tensor product $A\otimes_\C B$. In general there are also other completions but we will only need the maximal one in this article.

\begin{Prop}\label{prop tensor exact}
The maximal tensor product is exact, i.e. given a short exact sequence 
\[ \xymatrix{0 \ar[r] & J \ar[r] & A \ar[r] & B \ar[r] & 0}\]
and any other $C^*$-algebra $D$, then the induced sequence
\[ \xymatrix{0 \ar[r] & J\otimes D \ar[r] & A\otimes D \ar[r] & B\otimes D \ar[r] & 0 }\]
is exact as well. 
\end{Prop}
\begin{proof}
This is the content of \cite[II.9.6.6]{Blackadar2}.
\end{proof}
\begin{Rmk}
There is a notion of fibrations between $C^*$-algebras called \emph{Schochet fibrations}, cf. \cite[Definition 2.14]{Uuye}. With this notion \cref{prop tensor exact} has the addendum that for a Schochet fibration $A \to B$ and any $C^*$-algebra $D$, the induced map $A \otimes D \to B \otimes D$ is also a Schochet fibration, see \cite[Lemma 2.17]{Uuye}.
\end{Rmk}

Classical and powerful invariants of a $C^*$-algebra $A$ are the topological $K$-theory groups $K_*(A)$. For a compact Hausdorff space $X$ the groups $K_*(C(X))$ coincide with the usual topological $K$-theory group $K^{-*}(X)$ defined via vector bundles over $X$. For all $C^*$-algebras, topological $K$-theory is $2$-periodic, thanks to \emph{Bott periodicity}. See \cite{Rordam}, \cite{Blackadar}, \cite{Wegge-Olsen} for a treatment of the basics of $K$-theory for operator algebras.  

The topological $K$-groups of a $C^*$-algebra can be obtained as the homotopy groups of a $K$-theory spectrum $KA$, see for instance \cite{Meyer2} or \cite{Joachim}. More precisely, the group valued $K$-functor factors through (any of) the $1$-categories of spectra like symmetric or orthogonal spectra. An important feature of topological $K$-theory is that it is \emph{excisive}, in other words if
\[ \xymatrix{0 \ar[r] & J \ar[r] & A \ar[r] & B \ar[r] & 0}\]
is a short exact sequence of $C^*$-algebras, then the sequence
\[ \xymatrix{KJ \ar[r] & KA \ar[r] & KB}\]
is a fiber sequence of spectra. The long exact sequence it induces in homotopy groups is the usual long exact sequence of topological $K$-groups associated to the short exact sequence as above. 

The spectrum valued topological $K$-theory functor
admits the structure of a lax symmetric monoidal functor, i.e. for $A,B \in \cc$ there is a map of spectra
\[ KA \otimes KB \to K(A\otimes B) \]
satisfying the usual axioms of lax symmetric monoidal functors. Here we denote the smash product of spectra by $\otimes$. In \cite{Joachim} Joachim gave a point-set level description of this structure in orthogonal spectra.

\begin{Rmk}
From Joachim's point-set model, it follows that the induced functor
\[ \N\cc \to \Sp \]
of $\infty$-categories also admits a lax symmetric monoidal refinement which is all we will actually use in this article. Recall that $\Sp$ denotes the $\infty$-category of spectra. For the notion of lax symmetric monoidal functors between $\infty$-categories, see \cref{Multiplicative properties} and \cite[chapter 2]{LurieHA}. We will give an independent (equivalent) construction of the $K$-theory functor in \cref{K-theory-spectra}. The fact that $K$-theory admits a lax symmetric monoidal refinement is then a formal consequence of \cite[Corollary 6.8]{Nikolaus}.
\end{Rmk}

Kasparov's $\KK$-groups are of particular importance in the theory of operator algebras and their $K$-theory and the relation to index theory. $\KK$-theory and its applications to index theory (e.g. the Novikov conjecture) have been studied in \cite{KasparovI}, \cite{KasparovII}, \cite{KasparovIII}, and \cite{KasparovIV} and play a prominent role in the analytical aspects of the Baum-Connes conjecture.
The most important feature for us is the following theorem due to Kasparov. 
\begin{Thm}\label{KK-category}
There is a category $\KK$ with the following properties, see e.g.\ \cite{Blackadar} and \cite[Thm. 2.29 \& Rmk. 2.30]{Uuye}:
\begin{enumerate}
\item[(1)] $\Ob(\KK) = \Ob(\cc)$ and there is a functor $\cc \to \KK$, which we denote by $f \mapsto [f]$ on morphisms,
\item[(2)] This functor is a homotopy functor, i.e. if $f$ and $g$ are homotopic, then $[f] = [g]$.
\item[(3)] The category $\KK$ is triangulated and symmetric monoidal via the maximal tensor product. Exact sequences are short exact sequences of $C^*$-algebras in which the epimorphism is a Schochet fibration. The loop functor is the $C^*$-algebraic suspension functor.
\item[(4)] The abelian group valued $K$-functor $K \colon \cc \to \Ab$ factors through the functor $\cc \to \KK$ and the induced functor $\KK \to \Ab$ is corepresentable by the tensor unit object $\C$, i.e. there is an isomorphism of functors $K(-) \cong \Hom_\KK(\C,-)$.
\item[(5)] The abelian groups $\KK(A,B)$ can be described as equivalence classes of triples $(\mathcal{E},\pi,F)$, where $\mathcal{E}$ is a Hilbert-$B$-module, $\pi\colon A \to \mathcal{L}(\mathcal{E})$ is a representation and $F \in \mathcal{L}(\mathcal{E})$ satisfying certain compactness conditions.
\end{enumerate}
\end{Thm}

The property of $K$-theory described in $(4)$ is very useful to understand natural transformations $\tau \colon K \to F$ for some functor $F\colon \KK \to \Ab$ by virtue of the Yoneda lemma. It will be the main objective of \cref{KK infty category} to obtain a similar property for the spectrum valued $K$-functor.

\begin{Def}
A morphism $f\colon A \to B$ in $\cc$ is called a $\KK$-equivalence if its image in $\Hom_\KK(A,B)$ is an isomorphism.
\end{Def}
We want to conclude by stating a universal property of $\KK$.
\begin{Thm}\label{Uuye}
The functor $\cc \to \KK$ is a localization along the $\KK$-equivalences. In other words, a functor $F\colon \cc \to \Ab$ factors (necessarily uniquely) through $\KK$ if and only if $F$ has the property of sending $\KK$-equivalences to isomorphisms. These functors are characterized by the property of being split exact and stable.
\end{Thm}
\begin{proof}
Almost all of this is proven in \cite{Higson87}. The only thing missing is the fact that split exact and stable functors are automatically homotopy invariant. This was done in \cite{Higson88}.
We remark that more is known. In \cite[Theorem 2.29]{Uuye} it was proven that $\cc$ admits the structure of a fibration category, where the equivalences are the $\KK$-equivalences and the fibrations are the Schochet fibrations. Indeed, $\KK$ is the homotopy category of this fibration category.
\end{proof}

\begin{Rmk}
More precisely we have that the induced functor
\[ \Fun(\KK,\Ab) \to \Fun(\cc,\Ab)\]
is fully faithful and the image consists of those functors that are split exact and stable.
\end{Rmk}

%%%%%%%%%%%%%%%%%%%%%%%%%%%%%%%%%%%%%%%%%%%%%%%%%%%%%%%%%%%%%%

\subsection{$L$-theory of involutive rings}\label{L-theory}
In this section we want to recall basic notions from algebraic $L$-theory. 
$L$-theory has its origins in surgery theory, where $L$-groups appear as obstruction groups to deciding whether a given degree $1$ normal map between manifolds is bordant to a homotopy equivalence, see for instance \cite{Wall}, \cite{RanickiTSO}, and \cite{CLM}. Moreover there are connections to the algebraic theory of forms, relating $L$-groups to Witt groups of forms.
Algebraic $L$-theory has been developed by Ranicki in the series of papers \cite{Ranicki1}, \cite{Ranicki2}, \cite{Ranicki3}, \cite{Ranicki4} and the two books \cite{Ranicki} and \cite{RanickiBlue}.

For our purposes it is appropriate to view $L$-theory as a functor
\[\xymatrix@R=.2cm{\N\ringsinv \ar[r]^-{L} & \Sp \\ R \ar@{|->}[r] & LR}\]
whose construction is again due to Ranicki, see e.g. \cite[chapter 13]{RanickiBlue}. To be precise, Ranicki constructs a functor of $1$-categories
\[\xymatrix{\ringsinv \ar[r]^-{L} & \Sp_1}.\]
Applying the nerve to this functor and composing the result with the canonical map $\N\Sp_1 \to \Sp$ one obtains the desired $L$-theory functor.
If we want to be precise about the involution on $R$ (e.g. if there is more than one canonical involution) we will write $L(R,\tau)$. To be more specific, $LR$ is the projective symmetric algebraic $L$-theory spectrum associated to the ring $R$. We will usually not encounter the quadratic counterpart and thus simply write $LR$ for this spectrum. The spectrum $LR$ is constructed in a way such that its homotopy groups $\pi_n(LR)$ are the algebraic bordism groups of $n$-dimensional symmetric algebraic Poincar\'e complexes over $R$. 

\begin{Rmk}
If $P$ is a finitely generated projective module over $R$ and $\varphi\colon P \xrightarrow{\cong} P^*$ is a non-degenerate hermitian (respectively skew hermitian) form on $P$ then it gives rise to an element $[P,\varphi] \in L_0(R,\tau)$ respectively in $L_2(R,\tau)$. It is a theorem of Ranicki that if $2$ is invertible in the ring $R$ all elements in the algebraic $L$-groups $\pi_{2*}(LR)$ are of this form. 

More precisely if 
$R \in \ringsinv$ such that $2 \in R^\times$, then every $2n$-dimensional symmetric complex $(C,\varphi)$ over $R$ is bordant to a chain complex that is concentrated in degree $n$ and every $2n+1$-dimensional symmetric complex $(C,\varphi)$ is bordant to one that is concentrated in degrees $n$ and $n+1$.
Moreover the $L$-groups $\pi_*(LR)$ are isomorphic to the classical $L$-groups defined via forms and formations.

As indicated above there is a version called quadratic $L$-theory built on quadratic forms as opposed to hermitian forms. The previously mentioned result can be formulated in quadratic $L$-theory and is then true in full generality. The assumption that $2 \in R^\times$ is used to ensure that symmetric and quadratic $L$-theory are equivalent.
\end{Rmk}

\begin{Prop}\label{properties of L-theory}
$L$-theory satisfies the following properties.
\begin{enumerate}
\item[(1)] Algebraic $L$-theory is naturally $4$-periodic, i.e. $\Sigma^4(LR) \simeq LR$ for all involutive rings $R$. 
\item[(2)] If $-1$ has a square root $\alpha$ in $(R,\tau)$ which satisfies $\tau(\alpha) = -\alpha$, then $L$-theory becomes $2$-periodic. As an example $L\C = L(\C,x\mapsto \overline{x})$ is $2$-periodic, but $L(\C,\id)$ is \emph{not} $2$-periodic. To be specific we have that $\pi_*(L\C) = \Z[b^{\pm 1}]$ with $|b| = 2$ and $\pi_*(L(\C,\id)) = \F_2[t^{\pm1}]$ with $|t| =4$.
\item[(3)] $L$-theory commutes with finite products of involutive rings.
\item[(4)] Ranicki showed that $L$-theory admits external products $LS \otimes LT \to L(S\otimes T)$, more precisely the functor $L\colon \ringsinv  \to \mathcal{SHC}$ admits a lax symmetric monoidal refinement, where $\mathcal{SHC}$ denotes the stable homotopy category. In particular for every commutative ring $S$ the spectrum $LS$ is a ring spectrum, and for every $S$-algebra $T$, the spectrum $LT$ is a module spectrum over $LS$. In particular every spectrum $LR$ is a module over $L\Z$.
\item[(5)] Using the notion of ad-theories, in \cite{Laures} and \cite{Laures2} the authors establish a lax symmetric monoidal refinement of $L$-theory with values in the $1$-category of symmetric spectra. This implies that the induced functor
	\[ L \colon \N\ringsinv \to \Sp \]
of $\infty$-categories also admits a lax symmetric monoidal refinement. In particular it follows that the above monoidal properties not only hold in the homotopy category of spectra, but indeed in the $\infty$-category of spectra. Again for details on the notion and theory of symmetric monoidal $\infty$-categories we refer to \cref{Multiplicative properties} and \cite[chapter 2]{LurieHA}.
\end{enumerate}
\end{Prop}

%%%%%%%%%%%%%%%%%%%%%%%%%%%%%%%%%%%%%%%%%%%%%%%%%%%%%%%%%%%%%%

\subsection{$L$-theory of $C^*$-algebras.}\label{L-theory of CAlg}
Algebraic $L$-theory for $C^*$-algebras is defined by the composite
\[\xymatrix{\N\ccu \ar[r] & \N\ringsinv \ar[r]^-{L} & \Sp}.\]
This means that we will always take the involution coming from the $C^*$-algebra as input for $L$-theory. In particular $L\C$ is taken by using the complex conjugation as involution, not the identity.

Notice that so far we have not defined algebraic $L$-theory for \emph{non-unital} $C^*$-algebras, which is what we will do next. Recall from \cref{lemma unitalization} that for a $C^*$-algebra $A$ we have the associated split short exact sequence
\[\xymatrix{0 \ar[r] & A \ar[r] & A^+ \ar[r] & \C \ar[r] & 0}.\]

\begin{Def}
Let $A \in \cc$. We define its $L$-theory spectrum by the formula
\[ LA = \fib\left(L(A^+) \to L\C \right).\]
\end{Def}

\begin{Rmk}
If $A$ was unital then we have not changed the definition of the $L$-spectrum up to canonical equivalence since $L$-theory commutes with finite products. This property remains true on $\cc$, i.e. on non-unital $C^*$-algebras, see \cref{L-theory and products} and \cref{non unital L-theory}. It is worthwhile to compare the definition of non-unital $L$-theory to the remark after the proof of \cref{bladddd}.
\end{Rmk}

The following theorem is one of the crucial facts about the $L$-groups of unital $C^*$-algebras. 
\begin{Thm}\label{K groups isomorphic to L groups}
Let $A \in \cc$. Then there is a natural isomorphism
\[ \tau_A\colon K_n(A) \xrightarrow{\cong} L_n(A) \]
for all $n \in \Z$.
\end{Thm}
\begin{proof}
This is for instance proven in \cite{Miller}. Also \cite[Theorem 1.8]{Rosenberg} provides a proof. We recall briefly what we will use later. The first thing to notice is that it suffices to prove this for unital $A$. Furthermore both $K$- and $L$-theory are naturally $2$-periodic, so it suffices to prove the claim for $n = 0,1$. We will outline the $n=0$ case only.
The idea is as follows. Since $K_0(A)$ is the Grothendieck group of finitely generated projective modules over $A$ in order to construct a map 
\[ \tau_A\colon K_0(A) \to L_0(A) \]
it suffices to explain where to map the class of a finitely generated projective module $P$. It is a lemma of Karoubi, see \cite[Lemma 2.9]{Karoubi}, that any finitely generated projective $A$-module $P$ has a positive definite hermitian form $\sigma_P$ on it. Furthermore this form is unique up to isomorphism and this isomorphism may even be chosen to be homotopic to the identity. One way to construct $\sigma_P$ is to notice that any embedding $P \subseteq A^n$ gives $P$ the structure of a Hilbert-$A$-module by restricting the scalar product of $A^n$ to $P$. The association $[P] \mapsto [P,\sigma_P]$ is obviously compatible with taking direct sums and thus gives a map
\[\xymatrix{ K_0(A) \ar[r] & L_0(A) }\]
as claimed. One can use spectral theory in $C^*$-algebras to prove that any hermitian non-degenerate form is equivalent to the sum of a positive definite and a negative definite form. Using the uniqueness part of Karoubi's lemma one deduces that this construction is an inverse to the above map.
\end{proof}

\begin{Cor}\label{L-factors-through-KK}
$L$-theory is $\KK$-invariant, i.e. if $f\colon A \to B$ is a $\KK$-equivalence, then the induced morphism
\[ Lf \colon LA \to LB \]
is an equivalence of spectra. It follows that the $L$-groups may be viewed as a functor $\KK \to \Ab$.
\end{Cor}

%%%%%%%%%%%%%%%%%%%%%%%%%%%%%%%%%%%%%%%%%%%%%%%%%%%%%%%%%%%%%%

\subsection{KK-theory}
The idea of this section is to understand the natural transformation $\tau \colon K_0 \to L_0$ of \cref{K groups isomorphic to L groups} in terms of a universal property which will have a direct analogue in the case where we study the spectrum valued functors.

\begin{Lemma}
The canonical map 
\[\xymatrix{ \Hom_{\Fun(\KK,\Ab)}(K_0,L_0) \ar[r]^-\cong & \Hom_{\Fun(\cc,\Ab)}(K_0,L_0) }\]
is a bijection.
\end{Lemma}
\begin{proof}
This follows immediately from the remark after \cref{Uuye} and \cref{L-factors-through-KK}.
\end{proof}

Since $K$-theory becomes corepresentable on $\KK$, see \cref{KK-category} part (4), we obtain the following.

\begin{Cor}\label{tau second time}
From the enriched Yoneda lemma we see that
\[ \Hom_{\Fun(\KK,\Ab)}(K_0,L_0) \cong \Hom_{\Fun(\KK,\Ab)}(\KK(\C,-),L_0) \cong L_0(\C) \cong \Z.\]
In particular the transformation $\tau \colon K_0 \to L_0$ as given in \cref{K groups isomorphic to L groups} corresponds to an element of $\Z$.
Under this isomorphism $\tau$ is sent to $1 \in \Z$.
\end{Cor}
\begin{proof}
We only need to check the image of $\tau$ in $\Z$ under the above chain of isomorphisms.
By definition the isomorphism
\[\Hom_{\Fun(\KK,\Ab)}(K_0,L_0)\cong L_0(\C)\]
maps $\tau$ to $\tau_\C(\id)$ where
\[\tau_\C\colon \KK(\C,\C) \cong K_0(\C) \to L_0(\C).\]
Under the isomorphism $\KK(\C,\C) \cong K_0(\C)$ the element $[\id]$ is mapped to the projective module $\C$. The isomorphism of \cref{K groups isomorphic to L groups} takes this module to $\C$ equipped with the standard hermitian form over it. Furthermore the isomorphism $L_0(\C) \cong \Z$ takes the signature of this hermitian form which is clearly $1$.
\end{proof}

%%%%%%%%%%%%%%%%%%%%%%%%%%%%%%%%%%%%%%%%%%%%%%%%%%%%%%%%%%%%%%

\section{Construction of the map}\label{Construction of the map}

In this section we introduce the symmetric monoidal $\infty$-category $\KK_\infty$ and use this $\infty$-category to construct the natural map $kA \to \ell A$.

\subsection{The $\infty$-category $\KK_\infty$}\label{KK infty category}
We have argued how we can view the transformation from $K_0$ to $L_0$ using the universal property of the $\KK$ category. The main idea now is to mimic the universal properties we used, namely that $K$-theory is corepresentable on $\KK$. 
The following proposition is an important construction in $(\infty)$-categories, which we will use to define $\KK_\infty$.
\begin{Prop}\label{localization}
Suppose $\c$ is an $\infty$-category and $W$ is a collection of morphisms in $\c$. Then there is an $\infty$-category $\c[W^{-1}]$ and a functor of $\infty$-categories
\[ i\colon \c \to \c[W^{-1}] \] that is a \emph{Dwyer-Kan localization} along the morphisms in $W$, i.e. for every $\infty$-category $\Dd$ the functor
\[\xymatrix{\Fun(\c[W^{-1}],\Dd) \ar[r]^-{i^*} & \Fun(\c,\Dd)}\]
is fully-faithful and the image consists of those functors that send morphisms in $W$ to equivalences in $\Dd$.
\end{Prop}
\begin{proof}
This essentially goes back to \cite{Dwyer}. An argument in the language of $\infty$-categories is given for instance in \cite[Def. 1.3.4.1 and Rmk. 1.3.4.2]{LurieHA}. One can define the localization to be a fibrant replacement of the object $(\c,W)$ in the cartesian model structure on marked simplicial sets $\Set_\Delta^+$. 
\end{proof}
\begin{Rmk}
This universal property characterizes the $\infty$-category $\c[W^{-1}]$ up to equivalence.
\end{Rmk}

\begin{Def}\label{stableKKcategory}
We define the $\infty$-category $\KK_\infty$ to be the $\infty$-category obtained from the category $\cc$ by inverting the $\KK$-equivalences. In formulas we have that $\KK_\infty = \N\cc[W^{-1}]$ where $W$ denotes the collection of $\KK$-equivalences.
\end{Def}
\begin{Rmk}
The $\infty$-category $\KK_\infty$ is a full subcategory of the (presentable) $\infty$-category associated to the (combinatorial) model category of pro-$C^*$-algebras, see \cite{BJM}. The homotopy category of $\KK_\infty$ is equivalent to the category $\KK$ of \cref{KK-category}.
\end{Rmk}

The following proposition is well known and has been established in various different situations. For completeness we give an argument as well. For the notion and theory of stable $\infty$-categories, see \cite[chapter 1]{LurieHA}.
\begin{Prop}\label{KK-stable}
The $\infty$-category $\KK_\infty$ is a small stable $\infty$-category.
\end{Prop}
To prove this we prove the following preliminary lemma. A similar statement in the case where $\c$ is the $\infty$-category associated to a fibration category has also been made in \cite[section 3]{Cisinski}.
\begin{Lemma}\label{pointedness}
Let $\c$ be an $\infty$-category.
\begin{enumerate}
\item[(1)] If $\c$ admits all finite limits then $\c$ is pointed if its homotopy category $h\c$ is.
\item[(2)] Let $F\colon \c \to \Dd$ be a limit-preserving functor between pointed $\infty$-categories that admit all finite limits. Then $F$ is an equivalence if and only if $hF\colon h\c \to h\Dd$ is.
\end{enumerate}
\end{Lemma}
\begin{proof}
To show the first part let $\ast \in \c$ be a terminal object. We need to show that $\Map_{\c}(\ast,X)$ is contractible for all objects $X \in \c$.
The condition that $h\c$ is pointed implies that $\pi_0(\Map_{\c}(\ast,X)) = \{\ast\}$ for all $X\in \c$. Moreover, the diagram
\[\xymatrix{\Omega X \ar[r] \ar[d] & \ast \ar[d] \\ \ast \ar[r] & X}\]
is a pullback in the $\infty$-category $\c$. Thus we obtain that
\[ \Map_\c(\ast,\Omega X) \simeq \Omega \Map_\c(\ast,X)\]
and hence for all objects $X \in \c$ we get that
\[ \pi_n(\Map_\c(\ast,X)) \cong \pi_0(\Map_\c(\ast,\Omega^n X)) = \{\ast\}\]
so part $(1)$ follows.

To see the second part we recall that a functor is an equivalence if and only if it is essentially surjective and fully-faithful (meaning the induced map on mapping \emph{spaces} is an equivalence).
Essential surjectivity follows from the fact that $hF$ is an equivalence. Now we consider the induced map
\[\xymatrix{\Map_\c(X,Y) \ar[r] & \Map_{\Dd}(FX,FY).}\]
which is a bijection on $\pi_0$ for all objects $X,Y \in \c$ by the assumption that $hF$ is an equivalence. Using that $F$ preserves limits one shows that it also induces a bijection on $\pi_n$ for all $n \geq 0$ and all objects $X,Y \in \c$.
\end{proof}

\begin{proof}[Proof of \cref{KK-stable}]
First we notice that the category $\cc$ is small and thus the $\infty$-category $\N\cc$ is also small. It follows that the localization $\KK_\infty$ is small as well. Thus it remains to show that $\KK_\infty$ is stable.
Following \cite[Corollary 1.4.2.27]{LurieHA} we will show that $\KK_\infty$ admits all finite limits, is pointed and that the loop functor $\Omega\colon \KK_\infty \to \KK_\infty$ is an equivalence.

Since $\cc$ is a fibration category, see \cite[Theorem 2.29]{Uuye}, it follows by \cite[section 3]{Cisinski} that the simplicial localization admits a homotopy terminal object and homotopy pullbacks. Using \cite[Theorem 4.2.4.1]{LurieHTT} one deduces that $\KK_\infty$ admits a terminal object and pullbacks. It follows from the dual statement of \cite[Corollary 4.4.2.4]{LurieHTT} that $\KK_\infty$ thus admits all finite limits. Also consult \cite{Szumilo} for a discussion of the relation between (co)fibration categories and $\infty$-categories. 

To see that $\KK_\infty$ is pointed we apply \cref{pointedness}, using that $\KK$ is pointed.
Since $\Omega$ preserves all limits it follows that $\Omega\colon \KK_\infty \to \KK_\infty$ is an equivalence if it induces an equivalence on $\KK$, again recall \cref{pointedness}. This follows since $\KK$ is known to be triangulated with shift functor induced by $\Omega$, see e.g. \cite[Theorem 2.29]{Uuye} or \cref{KK-category} part (3).
\end{proof}

\begin{Def}
Recall that we denote by $\s$ the $\infty$-category of spaces, by $\Sp$ the $\infty$-category of spectra and by $\Sp_{\ge0} \subseteq \Sp$ the full subcategory of connective spectra.
There are canonical maps 
\[\xymatrix{\Sp \ar[r]^-{\tau_{\geq0}} & \Sp_{\geq 0} \ar[r]^-{\Omega^\infty} & \s}\]
which are all right adjoints.
For two $\infty$-categories $\Cc,\Cc'$ that admit finite products we denote the functor category of product preserving functors by $\Fun^{\Pi}(\c,\c')$, and if $\c,\c'$ admit all finite limits we denote the functor category of limit-preserving functors by $\Fun^{\lex}(\c,\c')$.
\end{Def}

We need the following general lemma about $\infty$-categories.

\begin{Lemma}\label{Lemma:Yoneda}
Let $\Dd$ be a stable resp. additive $\infty$-category. Then the Yoneda embedding admits an essentially unique refinement as indicated in the following diagram.
\[\xymatrix@R=.7cm{\Dd \ar[r] \ar@/_1.5pc/@{-->}[dr] & \Fun^\lex(\Dd^{op},\s) & \mathrm{resp.} & \Dd \ar[r] \ar@/_1.5pc/@{-->}[dr] & \Fun^\Pi(\Dd^{op},\s)  \\ 
& \Fun^\lex(\Dd^{op},\Sp) \ar[u]_{\Omega^\infty}^\simeq  & & & \Fun^\Pi(\Dd^{op},\Sp_{\geq0}) \ar[u]_{\Omega^\infty}^\simeq }\]
We will write $\underline{X}$ for the image of an object $X \in \Dd$ under these functors.
Given any object $X \in \Dd$ and any functor $F \in \Fun^\lex(\Dd,\Sp)$, resp. $F' \in \Fun^\Pi(\Dd,\Sp_{\geq0})$, there are equivalences
\[ \xymatrix@R=.3cm{\Map_{\Fun(\Dd,\Sp)}(\underline{X},F) \ar[r]^-\simeq & \Omega^\infty FX \\ 
		\Map_{\Fun(\Dd,\Sp_{\geq 0})}(\underline{X},F') \ar[r]^-\simeq & \Omega^\infty F'X }\]
of spaces.
\end{Lemma}
\begin{proof}
The first part follows from the fact that the maps
\[ \xymatrix@R=.3cm{\Fun^\lex(\Dd,\Sp) \ar[r]^-{\Omega^\infty} & \Fun^\lex(\Dd,\s) \\ 
		\Fun^\Pi(\Dd,\Sp_{\geq0}) \ar[r]^-{\Omega^\infty} & \Fun^\Pi(\Dd,\s) }\]
are equivalences, which follows from \cite[Corollary 1.4.2.23]{LurieHA} in the stable case and from \cite[Corollary 2.10]{GGN} in the additive case. Notice that there is an equivalence $\Grp_{\e_\infty}(\s) \simeq \Sp_{\geq 0}$.
For the second part we recall from \cite[Lemma 5.5.2.1]{LurieHTT} that for all $G \in \Fun(\Dd,\s)$, there is an equivalence
\[\xymatrix{ \Map_{\Fun(\Dd,\s)}(\underline{X},G) \ar[r] & GX }\]
which by the above equivalences finish the proof of the lemma.
\end{proof}

\begin{Rmk}
In an informal way this encodes that a stable (additive) $\infty$-category is \emph{enriched} in the stable $\infty$-category of spectra (the additive $\infty$-category of connective spectra).
\end{Rmk}

Since we know that both $K$ and $L$-theory can be viewed as functors $\N\cc \to \Sp$ we want to argue that we can view them as functors on $\KK_\infty$.

\begin{Prop}\label{K-theory-spectra}
The following statements hold true.
\begin{enumerate}
\item[(1)] The corepresented functor $\map_{\KK_\infty}(\C,-)\colon \KK_\infty \to \Sp$ is equivalent to $K$-theory,
\item[(2)] The functor $L \colon \N\cc \to \Sp$ factors through $\KK_\infty$.
\end{enumerate}
\end{Prop}
\begin{proof}
We prove the first statement first. Obviously the functor 
\[ K\colon \N\cc \to \Sp \]
sends $\KK$-equivalences to equivalences. Thus by the universal property of the localization functor $\N\cc \to \KK_\infty$ it follows that $K$-theory factors through $\KK_\infty$. We want to argue that $K$-theory can be viewed as an object of $\Fun^\lex(\KK_\infty,\Sp)$. This is because a functor preserves finite limits if and only if it preserves pullbacks and the terminal object, see \cite[Corollary 4.4.2.5]{LurieHTT}. Clearly $K$-theory preserves the terminal object and pullbacks because it preserves fiber sequences.
Thus from \cref{Lemma:Yoneda} we obtain an equivalence
\[ \Map_{\Fun(\KK_\infty,\Sp)}(\map_{\KK_\infty}(\C,-),K) \simeq \Omega^\infty (K\C).\]
We may thus consider the commutative diagram
\[ \xymatrix{\pi_0\left(\Map_{\Fun(\KK_\infty,\Sp)}(\map_{\KK_\infty}(\C,-),K) \right)\ar[r] \ar[d]_{\pi_0}& \pi_0\left(\Omega^\infty (K\C) \right) \ar[d]^-{\pi_0} \\ \Hom_{\Fun(\KK,\Ab)}(\KK(\C,-),K_0) \ar[r] & K_0(\C)}\]
where both horizontal arrows are isomorphisms by the corresponding Yoneda lemma.
We have already argued that the element $1 \in K_0(\C)$ comes from some isomorphism between the corepresented functor and the $K_0$-theory functor on $\KK$. By the fact that the right vertical arrow and top horizontal arrow are isomorphisms it follows that there exists a transformation $\eta\colon \map_{\KK_\infty}(\C,-) \to K$ that induces an isomorphism
\[ \eta_A \colon \pi_0(\map_{\KK_\infty}(\C,A)) \xrightarrow{\cong} \pi_0(KA) \]
for all $C^*$-algebras $A$. Since $\eta$ is a transformation between exact functors it follows that the diagram
\[ \xymatrix{\pi_n(\map_{\KK_\infty}(\C,A)) \ar[r] \ar[d]_\cong & \pi_n(KA) \ar[d]^\cong \\ \pi_0(\map_{\KK_\infty}(\C,S^n A)) \ar[r]_-\cong & \pi_0(K(S^n A))}\]
commutes. The lower horizontal map is an isomorphism by the previous argument, and both vertical maps are isomorphisms since the functors are exact. Thus $\eta$ is an equivalence as claimed.

In order to prove $(2)$ we need to see that $L$-theory sends $\KK$-equivalences to weak equivalences of spectra. This was established in \cref{L-factors-through-KK}.
\end{proof}

%%%%%%%%%%%%%%%%%%%%%%%%%%%%%%%%%%%%%%%%%%%%%%%%%%%%%%%%%%%%%%

\subsection{The construction of the natural map}\label{construction}

We are now in the position to prove the following theorem.
\begin{Thm}\label{space of transformations}
For every $n \in \Z$ there exists a natural transformation $\tau(n) \colon k \to \ell$, unique up to homotopy, characterized by the property that $\tau(n)_\C \colon \pi_0(ku) \to \pi_0(\ell\C)$ is given by multiplication by $n$. More precisely the map
\[ \xymatrix@R=.3cm{\pi_0\big( \Map_{\Fun(\N\cc,\Sp_{\geq 0})}(k,\ell) \big) \ar[r] & \Z \\ [\eta] \ar@{|->}[r] & \pi_0(\eta_\C)(1)}\]
is a bijection.
\end{Thm}
\begin{proof}

We have seen that $L$-theory factors over $\KK_\infty$ and thus want to appeal to the Yoneda lemma  to calculate transformations from $K$-theory to $L$-theory. 
Thus to obtain a natural transformation between $K$- and $L$-theory (viewed as connective spectra-valued functors) we need that the functor $\ell\colon \KK_\infty \to \Sp_{\geq 0}$ commutes with finite products, see \cref{Lemma:Yoneda}. 
Recall form \cref{localization} that 
\[\Fun(\KK_\infty,\Sp) \to \Fun(\N\cc,\Sp)\] 
is fully faithful. Since the map $\N\cc \to \KK_\infty$ preserves products also $\Fun^\Pi(\KK_\infty,\Sp)$ is a full subcategory of $\Fun^\Pi(\N\cc,\Sp)$. Thus it suffices to see that $L\colon \N\cc \to \Sp$ preserves finite products.
Recall that this is well known for unital $C^*$-algebras. In \cref{L-theory and products} we show that it also holds for non unital $C^*$-algebras. 

Hence we have the following chain of equivalences
\[\xymatrix{\Map_{\Fun(\cc,\Sp_{\geq 0})}(k,\ell) & \Map_{\Fun(\KK_\infty,\Sp_{\geq 0})}(k,\ell) \ar[l]_-\simeq \ar[r]^-\simeq & \Omega^\infty(\ell\C) }\]
where the first equivalence follows from \cref{localization} and the second equivalence is precisely \cref{Lemma:Yoneda}.
We get that
\[ \pi_0\big(\Map_{\Fun(\KK_\infty,\Sp_{\geq 0})}(k,\ell)\big) \xrightarrow{\cong} \pi_0(\ell\C) \cong \Z.\]
Furthermore the diagram
\[\xymatrix{\pi_0\big(\Map_{\Fun(\KK_\infty,\Sp_{\geq 0})}(k,\ell)\big) \ar[r] \ar[d] & \pi_0(\ell\C) \ar[d] \ar[r] & \Z \ar[d] \\
		\Hom_{\Fun(\KK,\Ab)}(K_0,L_0) \ar[r] & L_0(\C) \ar[r] & \Z }\]
commutes. This proves the theorem.
\end{proof}
We remark that the proof implies that
\[\tau = \tau(1) \in \Map_{\Fun(\KK_\infty,\Sp_{\geq 0})}(k,\ell),\]
is a transformation whose effect on $\pi_0$ is the transformation described in \cref{K groups isomorphic to L groups}. 

%%%%%%%%%%%%%%%%%%%%%%%%%%%%%%%%%%%%%%%%%%%%%%%%%%%%%%%%%%%%%%

\subsection{Multiplicative properties}\label{Multiplicative properties}

In this section we want to investigate the multiplicative properties of the two functors $K$- and $L$-theory and how the natural map between them respects multiplicative structures. 
We recall the definition of a symmetric monoidal structure on an $\infty$-category, see \cite[chapter 2]{LurieHA}.
\begin{Def}
A symmetric monoidal $\infty$-category is a coCartesian fibration $\c^\otimes \xrightarrow{p} \Nfin$ such that the Segal maps induce equivalences
\[\xymatrix{ \c^\otimes_{\langle n \rangle} \ar[r]^-{\simeq} & \prod\limits_n \c^\otimes_{\langle 1 \rangle} }\]
for all $\langle n \rangle \in \Nfin$. The fiber $\c^\otimes_{\langle 1 \rangle}$ over $\langle 1 \rangle$ will be denoted by $\c$ and referred to as the underlying $\infty$-category.
For two symmetric monoidal $\infty$-categories $\c^\otimes \to \Nfin $ and $\Dd^\otimes \to \Nfin$ we write $\Fun_\otimes(\c,\Dd)$ for the $\infty$-category of symmetric monoidal functors from $\c$ to $\Dd$. Those are just functors $\c^\otimes \to \Dd^\otimes$ over $\Nfin$ such that all coCartesian lifts are carried to coCartesian lifts.
We write $\Fun_\lax(\c,\Dd)$ for the $\infty$-category of lax symmetric monoidal functors from $\c$ to $\Dd$. Those are the functors $\c^\otimes \to \Dd^\otimes$ over $\Nfin$ such that coCartesian lifts of inert morphisms in $\Nfin$ are carried to coCartesian lifts. 
\end{Def}

As discussed in \cref{K-theory} the category $\cc$ admits a symmetric monoidal structure given by the maximal tensor product of $C^*$-algebras. It follows that the $\infty$-category $\N\cc$ admits a symmetric monoidal structure in the above sense, cf.\ \cite[Remark 2.0.0.6]{LurieHA}. We want to argue that this extends to a symmetric monoidal structure on $\KK_\infty$. For this we invoke the general theory of \cite[Section 3]{Hinich}. For convenience we recall the setup. 
\begin{Def}\label{SM localization}
Let $\c^\otimes \xrightarrow{p} \Nfin$ be a symmetric monoidal $\infty$-category. Suppose that $W \subseteq \c$ is a collection of morphisms in the underlying $\infty$-category $\c$. We define a collection $W^\otimes \subseteq \c^\otimes$ of morphisms as follows. We say $f$ is in $W^\otimes$ if
\begin{enumerate}
\item we have that $p(f) = \id_{\langle n \rangle}$ for some $\langle n \rangle \in \Nfin$ and
\item under the equivalence $\c^\otimes_{\langle n \rangle} \simeq \prod\limits_n \c$ the morphism $f$ corresponds to a tuple $(f_1,\dots,f_n)$ such that all $f_i$ belong to $W$.
\end{enumerate}
We say that the collection $W$ of morphisms in $\c$ is compatible with the monoidal structure if the tensor bifunctor (which is essentially uniquely determined by $\c^\otimes \to \Nfin$)
\[\otimes\colon \c \times \c \to \c \]
preserves morphisms in $W$ in both variables separately.
\end{Def}

\begin{Rmk}
Recall that the tensor bifunctor is given by 
\[ \mu_!\colon \c^2 \simeq \c^\otimes_{\langle 2 \rangle} \to \c^\otimes_{\langle 1 \rangle} = \c  \]
where $\mu\colon \langle 2 \rangle \to \langle 1 \rangle$ is the unique active morphism. 
Now suppose that $W$ is a collection of morphisms compatible with the monoidal structure in $\c$ so that $\mu_!$ preserves morphisms in $W^\otimes$. Then it follows formally that the same is true for any active morphism: if $\alpha \colon \langle n \rangle \to \langle m \rangle$ is active, then the induced functor
\[ \alpha_! \colon \c^\otimes_{\langle n \rangle} \to \c^\otimes_{\langle m \rangle} \]
preserves morphisms in $W^\otimes$.
We notice that for an inert morphism $\beta \colon \langle k \rangle \to \langle l \rangle \in \Nfin$ the induced functor
\[ \beta_! \colon \c^\otimes_{\langle k \rangle} \to \c^\otimes_{\langle l \rangle} \]
preserves morphisms in $W^\otimes$ by definition of $W^\otimes$. Thus $W$ is compatible with the monoidal structure in the sense of \cref{SM localization} if and only if for \emph{every} morphism $\alpha \in \Nfin$ the induced functor $\alpha_!$ preserves morphisms in $W^\otimes$.
\end{Rmk}

Let $\c^\otimes \to \Nfin$ be a symmetric monoidal $\infty$-category and let $W \subseteq \c$ be a collection of morphisms in the underlying $\infty$-category which is compatible with the monoidal structure in the sense of 
\cref{SM localization}. 
We will denote the $\infty$-category $\c^\otimes[(W^\otimes)^{-1}]$ by $\c[W^{-1}]^\otimes$.  By definition it comes with a canonical map $\c[W^{-1}]^\otimes \to \Nfin$. The following result is due to Hinich \cite{Hinich}, more specifically  see Proposition 3.2.2 and the remark after Definition 3.3.1 therein.

\begin{Prop}\label{localizing symmetric monoidal categories}
$\c[W^{-1}]^\otimes \to \Nfin$ is a symmetric monoidal $\infty$-category and the localization map  $i \colon\c^\otimes \to \c[W^{-1}]^\otimes$ is symmetric monoidal. Moreover the underlying $\infty$-category of  $\c[W^{-1}]^\otimes$ is equivalent to $\c[W^{-1}]$. 
Given any symmetric monoidal $\infty$-category $\Dd^\otimes \to \Nfin$ we have that the restriction map
\[ \Fun_\otimes(\c[W^{-1}],\Dd) \xrightarrow{i^*} \Fun_\otimes^W(\c,\Dd) \]
is an equivalence. Here the superscript $W$ refers to functors whose induced functor on underlying $\infty$-categories sends $W$ to equivalences. Similarly we also have that
\[ \Fun_\lax(\c[W^{-1}],\Dd) \xrightarrow{i^*} \Fun_\lax^W(\c,\Dd) \]
is an equivalence.
\end{Prop}

We apply this construction to the category $\cc$ with the maximal tensor product as symmetric monoidal structure and obtain
\begin{Cor}\label{KK symmetric monoidal}
There is a symmetric monoidal structure $\KK^\otimes_\infty$ on $\KK_\infty$ such that the localization map $i\colon\N\cc \to \KK_\infty$ admits a symmetric monoidal refinement. Moreover, for any other symmetric monoidal $\infty$-category $\Dd^\otimes \to \Nfin$ we have that the functor
\[ \Fun_\lax(\KK_\infty,\Dd) \xrightarrow{i^*} \Fun_\lax^W(\N\cc,\Dd) \]
is an equivalence where the superscript $W$ refers to functors whose underlying functor sends $\KK$-equivalences to equivalences.
\end{Cor}

Recall from \cite[Definition 4.1]{Nikolaus} that a stable and symmetric monoidal $\infty$-category is called stably symmetric monoidal if the tensor bifunctor is exact in both variables.
\begin{Lemma}\label{Lemmastably}
The symmetric monoidal $\infty$-category $\KK^\otimes_\infty$ is stably symmetric monoidal.
\end{Lemma}
\begin{proof}
It suffices to prove that the tensor bifunctor preserves finite sums and fiber sequences separately in both variables. For finite sums this is clear since tensoring with a sum of algebras is a sum of the tensor products. 
To get the result for fiber sequences we use that the existence of the fibration category structure on $\cc$ whose fibrations are the Schochet fibrations implies that 
every fiber sequence in $\KK_\infty$ is equivalent to one induced by a short exact sequence $I \to A \xrightarrow{p} B$ where $p$ is Schochet fibration. Thus if we tensor with some $D$ we obtain a sequence 
which is still short exact by
 \cref{prop tensor exact} and the map $D \otimes p$ is still a Schochet fibration by the remark after \cref{prop tensor exact}. Thus it follows that it is still a fiber sequence in $\KK_\infty$.
\end{proof}

\begin{Cor}\label{corlaxstructure}
$K$-theory $K \in \Fun(\KK_\infty,\Sp)$ has a canonical lax symmetric monoidal refinement. With this structure it is initial in $\Fun^{\lex}_\lax(\KK_\infty,\Sp)$. Similarly, connective $K$-theory is initial in  $\Fun^\Pi_{\lax}(\KK_\infty,\Sp_{\geq 0})$.
\end{Cor}
\begin{proof}
This is \cite[Corollary 6.8]{Nikolaus} using that $K$-theory is equivalent to the functor corepresented by the tensor unit object $\C$, recall \cref{K-theory-spectra}, and the fact that $\KK_\infty^\otimes$ is stably symmetric monoidal as shown in \cref{Lemmastably}.
\end{proof}

The following statement has been studied in different contexts, see \cite{Lurie} and \cite{Laures} in the case of unital rings. We refer to \cref{non unital L-theory}, more specifically \cref{L theory non unital monoidal}, where we argue how to deduce the following proposition from the the results of \cite{Laures2}.
\begin{Prop}\label{propositionlaxL}
The functors $L\in \Fun(\N\cc,\Sp)$ and $\ell\in \Fun(\N\cc,\Sp_{\geq 0})$ admit lax symmetric monoidal refinements.
\end{Prop}

\begin{Rmk}
In the appendix we prove the statement about the non-connective $L$-theory functor. Since the functor $\tau_{\geq 0} \colon \Sp \to \Sp_{\geq 0}$ is itself lax symmetric monoidal this implies the second part of \cref{propositionlaxL}.
\end{Rmk}

\begin{Cor}
The functors $L \in \Fun(\KK_\infty,\Sp)$ and $\ell \in \Fun(\KK_\infty,\Sp_{\geq 0})$ admit lax symmetric monoidal refinements.
\end{Cor}
\begin{proof}
This follows by \cref{propositionlaxL} and \cref{localizing symmetric monoidal categories}.
\end{proof}

\begin{Thm}\label{lax-refinement}
There is an essentially unique multiplicative transformation $\tau \colon k \to \ell$. Its underlying transformation is equivalent to $\tau(1)$ as constructed in \cref{space of transformations}. \end{Thm}
\begin{proof}
Since $k \in \Fun_\lax(\KK_\infty,\Sp_{\geq 0})$ is initial we see that there is an essentially unique element 
\[ \tau \in \Map_{\Fun_\lax(\KK_\infty,\Sp_{\geq 0})}(k,\ell).\]
The canonical map
\[ \xymatrix{\Map_{\Fun_\lax(\KK_\infty,\Sp_{\geq 0})}(k,\ell) \ar[r] & \Map_{\Fun(\KK_\infty,\Sp_{\geq 0})}(k,\ell)}\]
takes this to the natural transformation $\tau(n)$ for some $n$, see \cref{space of transformations}. The number $n$ is determined by its effect on
\[ \Z \cong \pi_0(ku) \longrightarrow \pi_0(\ell \C)\cong \Z,\]
namely $\tau(n)$ induces multiplication with $n$.
Since $\tau$ is multiplicative it follows that $\tau(n)$ is a map of rings and thus that $n=1$.
\end{proof}

\begin{Cor}\label{equivalence for C}
The map $\tau_\C \colon ku \to \ell\C$ is a map of $\e_\infty$-ring spectra.
\end{Cor}
\begin{Rmk}
Composing the map $\tau_\C$ with the complex orientation $\MU \to ku$ yields a complex orientation of $\ell\C$ and thus also of $L\C$. This orientation is very different from the one obtained through the Sullivan-Ranicki orientation $\MU \to \MSO \to L\Z \to L\C$. Using the computation of the map $\tau_\C$ in the next section one can deduce that the formal group law  of the new orientation is given by $x + y + 2xy$. The other orientation leads to a more complicated formal group law (e.g. it is a non finite power series) but of course the two are related by a change of coordinates. It is immediately clear that the formal group law $x + y + 2xy$ is isomorphic to the multiplicative formal group law when localized away from 2. Writing down the logarithm shows that after localizing at 2 this group law is isomorphic to the additive formal group law. This algebraic observation is compatible with the fact that at the prime 2 the spectrum $L\C$ is a generalized Eilenberg-Mac Lane spectrum and away from 2 it is equivalent to $K$-theory. Notice that this implies in particular that the spectrum $L\C$ is not Landweber exact.
\end{Rmk}

%%%%%%%%%%%%%%%%%%%%%%%%%%%%%%%%%%%%%%%%%%%%%%%%%%%%%%%%%%%%%%
%%%%%%%%%%%%%%%%%%%%%%%%%%%%%%%%%%%%%%%%%%%%%%%%%%%%%%%%%%%%%%

\section{The effect on homotopy groups}\label{effect on homotopy}

The main goal of this section is to calculate the effect of the multiplicative transformation $\tau\colon k \to \ell$ constructed in \cref{lax-refinement}  on homotopy groups. If $M$ is an abelian group and $n \geq 0$ we denote by $M_n$ the subgroup of $n$-torsion elements. We will prove the following theorem.

\begin{Thm}\label{effect-tau-on-homotopy}
For $i \in \{0,1\}$, all $k \geq 0$, and all $A \in \cc$ there is an exact sequence
\[\xymatrix{0 \ar[r] & \pi_{2k+i}(kA)_{2^k} \ar[r] & \pi_{2k+i}(kA) \ar[r]^-{\tau_A} & \pi_{2k+i}(\ell A) \ar[r] & \frac{\pi_{2k+i}(\ell A)}{2^k\cdot\pi_{2k+i}(\ell A)} \ar[r] & 0}.\]
\end{Thm}

The other maps $\tau(n): k \to \ell$  constructed in \cref{space of transformations} are given by $n \cdot \tau$, thus it is easy to deduce similar results for those.

The main method for the calculations we need is to understand how close the canonical map
\[ \Theta_A \colon \Sigma L(SA) \to LA \]
is from being an equivalence.

\subsection{Excision in $L$-theory}\label{excision}
Assume we have a diagram of rings
 \[\xymatrix{R \ar[r] \ar[d] & S \ar[d]^{p} \\ T \ar[r]_{q} & U}\]
which is a pullback and in which either $p$ or $q$ is surjective. 
In \cite[Theorem 3.3]{Milnor} it is shown that a pullback square as above  induces a long exact Mayer-Vietoris sequence in algebraic $K$-theory connecting degree $1$ to degree $0$, i.e. that there is an exact sequence
\[ K_1^\a(R) \to K_1^\a(S)\oplus K_1^\a(T) \to K_1^\a(U) \to K_0^\a(R) \to K_0^\a(S)\oplus K_0^\a(T) \to K_0^\a(U) .\]
Using Bass' definition of negative $K$-groups it follows formally that this sequence can be extended to the right by negative $K$-groups, see \cite[III Theorem 4.3]{Weibel}.

A crucial property of $L$-theory is that it does not admit such Mayer-Vietoris sequences in general. But there is a formula calculating the defect. This is the content of the following discussion which in similar forms has been investigated by Ranicki \cite[Chapter 6]{Ranicki} and Weiss-Williams \cite{WW-Duality}. The setup is as follows.
Let 
\[\xymatrix{R \ar[r] \ar[d] & S \ar[d]^p \\ T \ar[r]_q & U}\]
be a pullback square of involutive not necessarily unital rings where again, say, $p$ is surjective. We define the algebraic $K$-theory of a non-unital ring $R$ to be
\[ K_0^\a(R) = \ker\big( K_0^\a(R^+) \to K_0^\a(\Z) \big) \cong \coker \big( K_0^\a(\Z) \to K_0^\a(R^+)\big).\]
and let
\[ X = \coker \big(K_0^\a(S)\oplus K_0^\a(T) \to K_0^\a(U)\big).\]
We can now consider the commutative square of spectra
\[\xymatrix{LR \ar[r] \ar[d] & LS \ar[d] \\ LT \ar[r] & LU }\]
which produces an essentially unique map
\[ \xymatrix{LS \oplus_{LR} LT \ar[r] & LU }\]
where the left hand side denotes the pushout of the above diagram.
\begin{Thm}\label{excision}
In the situation of above, suppose that in addition $2$ is invertible in all rings. Then there is a fiber sequence 
\[\xymatrix{ LS\oplus_{LR} LT \ar[r] & LU \ar[r] & (HX)^{tC_2}}\]
of spectra, where $HX$ denotes the Eilenberg-MacLane spectrum on $X$.
\end{Thm}
\begin{Rmk}
First we notice that there is a natural $C_2$-action on the algebraic $K$-theory of an involutive ring. This action is induced by sending a projective left-module $P$ to its dual module, $\Hom_R(P,R)$ which one can view as a left-module via the involution on $R$. Since we declared the above square to be one of involutive rings it follows that there is an induced action of $C_2$ on the group $X$.
The superscript $tC_2$ in $(HX)^{tC_2}$ refers to the Tate construction of this group action. This is a spectrum with the property that the homotopy groups are given by classical Tate cohomology: 
\[ \pi_*((HX)^{tC_2}) \cong \hat{H}^{-*}(C_2;X).\]
\end{Rmk}

\begin{proof}[Proof of \cref{excision}]
First we show how to reduce the theorem to the case where all rings in question are unital. For that we look at the diagram obtained by unitalizing all rings. After applying $L$-theory we then obtain a commutative diagram of spectra
\[\xymatrix{ LR \ar[rr] \ar[dd] \ar[dr] & & LR^+ \ar[rr] \ar@{-->}[dd] \ar[dr] & & L\Z \ar@{-->}[dd] \ar[dr] & \\
		& LS \ar[rr] \ar[dd] & & LS^+ \ar[rr] \ar[dd] & & L\Z \ar[dd] \\
		LT \ar@{-->}[rr]  \ar[dr] & & LT^+ \ar@{-->}[rr] \ar@{-->}[dr] & & L\Z \ar@{-->}[dr] & \\
		& LU \ar[rr] & & LU^+ \ar[rr] & & L\Z }\]
Since the horizontal composites are all fiber sequences and the most right vertical square is a pullback it follows that the most left vertical square is a pullback if and only if the middle vertical square is a pullback. More precisely, comparing cofibers carefully one can see that the diagram
\[\xymatrix{LS\oplus_{LR} LT \ar[r] \ar[d] & LS^+\oplus_{LR^+} LT^+ \ar[d] \\ LU \ar[r] & LU^+}\]
is a pullback diagram. Thus if we assume the theorem in the case where all rings are unital we obtain a fiber sequence
\[\xymatrix{ LS\oplus_{LR} LT \ar[r] & LU \ar[r] & (HX)^{tC_2} }\]
where 
\[ X = \coker\big( K_0^\a(S^+) \oplus K_0^\a(T^+) \to K_0^\a(U^+) \big)\]
But by our definition of $K$-groups for non-unital rings it follows that
\[ X \cong \coker\big( K_0^\a(S) \oplus K_0^\a(T) \to K_0^\a(U)\big).\]
Thus it remains to prove the theorem in the case where all rings are unital.

One can then be more specific and even give a concrete description of the pushout $LS\oplus_{LR} LT$ in terms of $L$-theory. For an involutive ring and every $C_2$-invariant subgroup $Z \subseteq K_0^\a(R)$ there is a spectrum $L^Z(R)$ called $L$-theory with control in $Z$. Roughly speaking classes are represented by symmetric Poincar\'e complexes $(C,\varphi)$ whose underlying $K$-theory class $[C] \in K_0^\a(R)$ lies in $Z$, see e.g. \cite{Williams2} for details.
It is proven in \cite[6.3.1]{Ranicki} that in the situation of the theorem, with
\[ Z = \ker\big( K_0^\a(U) \to X \big) \]
there is a pullback diagram
\[ \xymatrix{LR \ar[r] \ar[d] & LS \ar[d] \\ LT \ar[r] & L^ZU }\]
and thus $LS\oplus_{LR} LT \simeq L^ZU$.
For this one needs to show that the diagram
\[\xymatrix{K_0^\a(R) \ar[r] \ar[d] & K_0^\a(S) \ar[d] \\ K_0^\a(T) \ar[r] & Z}\]
is cartesian in the sense of \cite[page 498]{Ranicki}, which is clear from the definition of $Z$. 
Notice that cartesian in Ranicki's sense does not mean that it is a cartesian square in abelian groups but rather that it is a cocartesian square of abelian groups.

We have now seen that the map $LS\oplus_{LR} LT \to LU$ is equivalent to the map
\[\xymatrix{ L^ZU \ar[r] & LU }\]
whose cofiber is equivalent to $H(K_0^\a(U)/Z)^{tC_2} \simeq (HX)^{tC_2}$ by the spectral Rothenberg sequence as described in \cite[page 27]{Williams2}. Thus the theorem follows.
\end{proof}

\begin{Cor}\label{split exact L-theory}
Consider a diagram of involutive rings in which $2$ is invertible
\[\xymatrix{R \ar[r] \ar[d] & S \ar[d] \\ T \ar[r] & U \ar@/_1pc/[u]}\]
which is a pullback and where the right vertical map admits a multiplicative section. Then the associated square of $L$-theory spectra is a pullback square.
\end{Cor}
\begin{proof}
If the map $p$ admits a multiplicative split it follows that the map $K_0^\a(S) \to K_0^\a(U)$ is surjective. Thus $X = 0$ and then \cref{excision}  directly implies the claim.
\end{proof}

\begin{Cor}\label{L-theory and products}
$L$-theory commutes with products in $\cc$.
\end{Cor}
\begin{proof}
Recall from \cref{properties of L-theory} that
$L$-theory commutes with products of unital rings and thus with products in $\ccu$. To show the general situation, we consider the diagram
\[\xymatrix{(A\times B)^+ \ar[r] \ar[d] & A^+\times B^+ \ar[d] \\ \C \ar[r]_-\Delta & \C\times \C \ar@/_1pc/[u]}\]
which is a pullback diagram as in \cref{split exact L-theory}. Thus the associated square of $L$-theory spectra is a pullback square and hence the canonical map of vertical fibers
\[\xymatrix{ L(A\times B) \ar[r] & LA \oplus LB }\]
is an equivalence as claimed.
\end{proof}

\begin{Cor}\label{L-theory exact after inverting 2}
Suppose that 
\[\xymatrix{R \ar[r] \ar[d] & S \ar[d]^p \\ T \ar[r]_q & U }\]
is a pullback square of involutive rings where $p$ is surjective. Then the square
\[\xymatrix{LR\adj \ar[r] \ar[d] & LS\adj \ar[d] \\ LT\adj \ar[r] & LU\adj}\]
is a pullback diagram.
\end{Cor}
\begin{proof}
This follows from the fact that Tate cohomology of $C_2$ is $2$-torsion and thus $(HX)^{tC_2}\adj \simeq 0$. The condition that $2$ is invertible in the rings is not needed. The reason is that \cref{excision} holds without the  assumption that $2$ is invertible if we work in quadratic $L$-theory $L^q$. But for every ring $R$ the canonical map 
\[ L^qR\adj \to LR\adj \]
is an equivalence, thus the statement also holds in symmetric $L$-theory after inverting 2.
\end{proof}

Let $A$ be a $C^*$-algebra. We can then consider the diagram
\[\xymatrix{SA \ar[r] \ar[d] & CA \ar[d]^p \\ 0 \ar[r] & A}\]
which is a pullback diagram in which the map $p$ is surjective. The algebra $CA$ is contractible as a $C^*$-algebra and thus $\KK$-equivalent to $0$. Since $L$-theory sends $\KK$-equivalences to equivalences it follows that $L(CA) \simeq 0$. Hence we get that
\[ 0 \oplus_{L(SA)} L(CA) \simeq 0 \oplus_{L(SA)} 0 \simeq \Sigma L(SA)\]
and obtain a canonical map
\[ \xymatrix{\Sigma L(SA) \ar[r]^-{\Theta_A} & LA}.\]

\begin{Prop}\label{cofiber of theta}\label{suspension-iso-L-theory}
Let $A$ be a $C^*$-algebra. Then there is a fiber  sequence
\[\xymatrix{\Sigma L(SA) \ar[r] & LA \ar[r] & H(K_0(A))^{tC_2}} .\]
The associated long exact sequence decomposes into two parts: an exact sequence
\[\xymatrix@C=.7cm{0\ar[r] & \hat{H}^{-2k-1}(C_2;K_0(A)) \ar[r] & \pi_{2k-1}(L(SA)) \ar[r] & \pi_{2k}(LA) \ar[r] & \hat{H}^{-2k}(C_2;K_0(A)) \ar[r] & 0 }\]
and an isomorphism
\[ \xymatrix{\pi_{2k}(L(SA)) \ar[r]^\cong& \pi_{2k+1}(LA) .} \]
\end{Prop}
\begin{proof}
The fiber sequence is a direct consequence of \cref{excision} using that for a $C^*$-algebra $A$ we have $K_0(A) \cong K_0^\a(A)$. To prove the second part we consider the following commutative square
\[
\xymatrix{ LA \ar[r] \ar[d] & H(K_0(A))^{tC_2} \ar[d]^\cong \\
		LA^+ \ar[r] & H(\wt{K}_0(A^+))^{tC_2}}
\]
where $\wt{K}_0(B)$ denotes reduced $K$-theory of a ring $B$. The upper horizontal fiber is  $\Sigma L(SA)$ by the first part. The lower horizontal fiber  
is what is called free $L$-theory of $A^+$ and denoted by $L^{\langle h\rangle}(A^+)$ following conventions from geometric topology. The associated long exact sequence is the classical Rothenberg exact sequence relating free and projective $L$-groups, see e.g. \cite[Proposition 1.10.1]{Ranicki}.
Then we obtain an induced diagram
\begin{align}\label{diagram1}
\xymatrix{\Sigma L(SA) \ar[r] \ar[d] & LA \ar[r] \ar[d] & H(K_0(A))^{tC_2} \ar[d]^\cong \\
		L^{\langle h \rangle}(A^+) \ar[r] & LA^+ \ar[r] & H(\wt{K}_0(A^+))^{tC_2}}
\end{align}
Notice that since the middle vertical map splits it follows that also the most left vertical map splits.
It is proven in \cite[Remark 1.7]{Rosenberg} that the classical Rothenberg sequence decomposes as described in the statement of the proposition.  This proof given in \cite{Rosenberg} relies on the statement that $\pi_0(KA) \cong \pi_0(LA)$, see \cref{K groups isomorphic to L groups}, and a similar analysis of the free $L$-group $L^{\langle h \rangle}_0(A)$ in terms of $K$-groups.
Using that the vertical maps in (\ref{diagram1}) split, a diagram chase then finishes the proof of the proposition.
\end{proof}

\begin{Cor}\label{times2}
For the algebra $A= \C$ we have a short exact sequence
\[\xymatrix{0 \ar[r] & L_1(S\C) \ar[r] & L_2(\C) \ar[r] & \Z/2 \ar[r] & 0}\]
\end{Cor}
\begin{proof}
By \cref{cofiber of theta} we only need to argue why $\hat H^{-3}(C_2, K_0(\C)) = 0$ and $\hat H^{-2}(C_2, K_0(\C)) = \Z/2$. Both facts  follow since $K_0(\C) \cong \Z$ carries the trivial $C_2$-action and the computation of Tate cohomology in this case. 
\end{proof}

\begin{Rmk}
The triviality of the $C_2$-action on $K_0(\C)$ holds in greater generality:
For a $C^*$-algebra $A$, all elements in $K_0(A)$ can be represented by (formal differences) of self-adjoint idempotents. Thus the action of $C_2$ on $K_0(A)$ is always trivial.
\end{Rmk}

\subsection{The proof of \cref{effect-tau-on-homotopy}}

The strategy will be to first prove \cref{effect-tau-on-homotopy} in the case $A= \C$ and then prove the general case using that the map $kA \to \ell A$ is a module map over $ku \to \ell \C$ and explicit low dimensional calculations. For this we need some preparatory lemmas.
\begin{Lemma}\label{effect-on-pi_0-and-pi_1}
The transformation $\tau$ satisfies that the map
\[ \tau_A \colon \pi_i(kA) \xrightarrow{\cong} \pi_i(\ell A) \]
is an isomorphism for all $A \in \KK_\infty$ and $i \in \{0,1\}$.
\end{Lemma}
\begin{proof}
The case $i=0$ follows since $\tau$ was chosen to lift the isomorphism of \cref{K groups isomorphic to L groups} on $\pi_0$, compare the remark after the proof of \cref{space of transformations}. For the case $i=1$ we consider the commutative diagram
\[\xymatrix{\pi_0(k(SA)) \ar[r]^-\cong \ar[d]_\cong & \pi_1(kA) \ar[d] \\ \pi_0(\ell(SA)) \ar[r]_-\cong & \pi_1(\ell A)}\]
where the horizontal maps come from the canonical maps of spectra
\[ \Sigma k(SA) \to kA \text{ and }  \Sigma \ell(SA) \to \ell A \]
which is an isomorphism on $\pi_0$ for (topological) $K$-theory (as $K$-theory is excisive) and it is the content of \cref{suspension-iso-L-theory} that it also an isomorphism for $L$-theory. Thus we deduce the claim from the case $i=0$.
\end{proof}

Using that both $\pi_2(ku)$ and $\pi_2(\ell\C)$ are infinite cyclic groups we now obtain the following 
\begin{Lemma}\label{effect-on-pi_2}
There is a generator $b \in \pi_2(\ell\C)$ such that the map $\tau_\C\colon ku \to \ell\C$ satisfies
\[ \xymatrix@R=.2cm{\pi_2(ku) \ar[r] & \pi_2(\ell\C) \\ \beta \ar@{|->}[r] & 2b }\]
\end{Lemma}
\begin{proof}
We consider the commutative diagram
\[\xymatrix{\pi_1(k(S\C)) \ar[r]^-\cong \ar[d]_\cong & \pi_2(ku) \ar[d] \\ \pi_1(\ell(S\C)) \ar[r]_-{\cdot 2} & \pi_2(\ell \C)}\]
The left vertical map is an isomorphism by \cref{effect-on-pi_0-and-pi_1}. The upper horizontal map is an isomorphism since $K$-theory is excisive and the fact that the lower horizontal map is given by multiplication by $2$ follows from \cref{times2}.
\end{proof}

\begin{proof}[Proof of \cref{effect-tau-on-homotopy}]
Recall from \cref{equivalence for C} that since $\tau$ is multiplicative it follows that the map $\tau_\C \colon ku \to \ell \C$ is a map of $\e_\infty$-ring spectra.
Thus \cref{effect-on-pi_2} and the fact that $ku_* = \Z[\beta]$ and $\ell\C_* = \Z[b]$ imply that 
\begin{align}\label{equation-beta} \tau_\C(\beta^k) = 2^k\cdot b^k.\end{align}
For $i \in \{0,1\}$ and $k \geq 0$ we consider the diagram 
\[\xymatrix{\pi_i(kA) \otimes \pi_{2k}(ku) \ar[r]^-\cong \ar[d]_{\tau_A\otimes \tau_\C} & \pi_{2k+i}(kA) \ar[d]^{\tau_A} \\ 
		\pi_i(\ell A) \otimes \pi_{2k}(\ell\C) \ar[r]_-\cong & \pi_{2k+i}(\ell A)}\]
which commutes because $\tau$ is lax symmetric monoidal (see \cref{lax-refinement}). Here the horizontal arrows come from the module structures of $kA$ and $\ell A$ over $ku$ and $\ell\C$ respectively and are isomorphisms by the respective periodicities. The proposition now follows  from \cref{effect-on-pi_0-and-pi_1} and \cref{equation-beta}.
\end{proof}

%%%%%%%%%%%%%%%%%%%%%%%%%%%%%%%%%%%%%%%%%%%%%%%%%%%%%%%%%%%%%%
%%%%%%%%%%%%%%%%%%%%%%%%%%%%%%%%%%%%%%%%%%%%%%%%%%%%%%%%%%%%%%

\section{Applications}\label{Applications}

\begin{Cor}\label{L adj exact}
The two functors $K\adj, L\adj: \cc \to \Sp$ are equivalent as lax symmetric monoidal functors.
\end{Cor}
\begin{proof}
We will use that
\[ KA \simeq kA[\beta^{-1}] \text{ and } LA \simeq \ell A[b^{-1}] \]
which follows from periodicity. From \cref{effect-on-pi_2} we see that the transformation $\tau_\C$ sends $\beta$ to $2b$. If we want to extend this map  periodically we thus have to invert 2 in the target. The fact that the kernel and cokernel of $\tau_A$ are 2-torsion as shown in \cref{effect-tau-on-homotopy} then implies that we get the desired equivalence after inverting 2.
\end{proof}

\begin{Rmk}
The fact that there is a natural equivalence $KA\adj \to LA\adj$ can also be proven without the explicit computation of the map $\tau_A$ on homotopy. 
We will do this in the real case in \cref{KO vs L} below.
The key ingredient is to see that $L\adj \in \Fun^\lex_\lax(\KK_\infty,\Sp)$ so that one can appeal to the stable Yoneda lemma.
\end{Rmk}

%%%%%%%%%%%%%%%%%%%%%%%%%%%%%%%%%%%%%%%%%%%%%%%%%%%%%%%%%%%%%%

\subsection{Real version}\label{real}

Let $\rc$ be the category of separable real $C^*$-algebras. It still has a forgetful functor
to the category of non-unital involutive rings.
Thus we may consider the functors
\[ KO, L \colon \N\rc \to \Sp \]
just as we have done before by replacing complex by real $C^*$-algebras throughout.
The main theorem we want to prove in this section is the following.
\begin{Thm}\label{KO vs L}
The two functors $KO\adj, L\adj\colon \N\rc \to \Sp$ are equivalent as lax symmetric monoidal functors.
\end{Thm}

This builds on two things. Firstly there is a version of the  $\KK$-category for real $C^*$-algebras, denoted by $\KK^\R$, with the same properties as in the complex case, i.e. it is a localization of $\rc$ along $\KK$-equivalences and $KO_0(-)$ becomes corepresentable on $\KK^\R$ by the tensor unit $\R$. Indeed, $\rc$ admits the structure of a fibration category such that the associated $\infty$-category $\KK_\infty^\R$ is stable with the same proof as in the complex case, cf.\ \cref{KK-stable}. It is also symmetric monoidal with respect to the maximal tensor product cf. \cref{KK symmetric monoidal}. Furthermore the functor $\map_{\KK^\R_\infty}(\R,-)$ is as a lax symmetric monoidal functor equivalent to $KO: \KK^\R_\infty \to \Sp$.

The second thing is that we need $L$-theory to factor through the real $\KK$-category. Recall that we deduced this in the complex case from the fact that the $L$-groups and $K$-groups are naturally isomorphic, see \cref{K groups isomorphic to L groups}. The corresponding statement is not true for real $C^*$-algebras as for example $KO_1(\R) \cong \Z/2$ but $L_1(\R) = 0$ shows. So we cannot conclude $\KK$-invariance for $L$-theory of real $C^*$-algebras as easily. But after inverting 2 it is still true that $K$ and $L$-groups are isomorphic as for example claimed in \cite[Theorem 1.11]{Rosenberg}. We will give an independent argument for $\KK$-invariance after inverting 2.

Let $A \in \rc$. We denote by $A_\C = A\otimes_\R \C$ its complexification which is a complex $C^*$-algebra. From the inclusion $A \to A_\C$ we obtain a natural map
$LA \to L(A_\C) . $

\begin{Prop}\label{L KK invariant}
After inverting 2 this map admits a natural retraction. It follows that the functor $A \mapsto LA\adj$ is $\KK$-invariant.
\end{Prop}
\begin{proof}
Recall that the $L$-theory of a ring $R$ is given by considering perfect complexes over $R$ with a symmetric structure. We claim that the inclusion map $A \to A_\C$ induces a map $L(A_\C) \to LA$ by restriction. 
Alternatively this map can be described as the composite
\[ L(A_\C) \to L(M_2(A)) \simeq LA \]
using the canonical embedding 
\[ \C \ni (x+iy) \mapsto  \begin{pmatrix} x & -y \\ y &x \end{pmatrix} \in M_2(\R).\]
and the fact that $L$-theory is Morita invariant. 
The composite
\[ LA \to L(A_\C) \to LA \]
obtained this way can be identified with multiplication by $2$. Thus the composite
\[ LA\adj \to L(A_\C)\adj \to LA\adj \]
is an equivalence.

For the second part we observe that if $f\colon A\to B$ is a real $\KK$-equivalence then by the first part the map
\[ Lf\adj \colon LA\adj \to LB\adj\]
is a retract of the map
\[ \xymatrix@C=2cm{L(A_\C)\adj \ar[r]^-{L(f_\C)\adj} & L(B_\C)\adj }\]
which is an equivalence since $f_\C\colon A_\C \to B_\C$ is a complex $\KK$-equivalence. 
\end{proof}

\begin{proof}[Proof of \cref{KO vs L}]
By \cref{L KK invariant} we see that $L$ theory induces a functor
\[ L\adj \in \Fun(\KK_\infty,\Sp).\]
We thus want to argue that $L\adj$ admits a lift along the map
\[ \Fun^\lex_\lax(\KK_\infty,\Sp) \to \Fun(\KK_\infty,\Sp).\]
To show that $L\adj$ is (left) exact we need to show that it commutes with finite limits. By \cite[Corollary 4.4.2.5]{LurieHTT} it suffices to show that $L\adj$ preserves the terminal object and pullbacks. It is clear that $L\adj$ preserves the terminal object and \cref{L-theory exact after inverting 2} shows that $L\adj$ commutes with pullbacks.
As in \cref{propositionlaxL} we see that $L$-theory admits a lax symmetric monoidal refinement. 

Thus it follows as in \cref{lax-refinement} that there is an essentially unique lax symmetric monoidal transformation $\tau\colon KO \to L\adj$. For the real $C^*$-algebra $\R$ this map induces on $\pi_0$  the unique ring map
\[ \Z \cong \pi_0(KO) \to \pi_0(L\R\adj) \cong \Z\adj.\]
As in the proof of \cref{K groups isomorphic to L groups} one can show that there is a natural isomorphism $KO_0(A) \to L_0(A)$ for all real $C^*$-algebras. The Yoneda lemma for $\KK^\R$ implies that the effect on $\R$ determines a natural transformation $KO_0 \to L_0$  uniquely.  Thus after inverting 2 this natural isomorphism has to agree with $\tau$ because they agree on $\R$. This shows  that the map
\[ \tau_A \colon \pi_0(KO(A)\adj) \to \pi_0(LA\adj) \]
is an isomorphism for all $A$. Finally it follows from excisiveness of both sides and the $\pi_0$-case that the induced transformation
\[ \tau \colon KO\adj \to L\adj \]
is an equivalence as claimed.
\end{proof}

\begin{Cor}\label{KOadj and LRadj E-infty equivalent}
The map 
\[ \tau_\R \colon KO\adj \to L\R\adj \]
is a map of $\e_\infty$-ring spectra and an equivalence, thus an equivalence of $\e_\infty$-ring spectra.
\end{Cor}

\begin{Rmk}
The fact that these two  spectra are equivalent as homotopy ring spectra was already observed in \cite[Lecture 25]{Lurie} by comparing the formal groups associated to these spectra. The equivalence of underlying spaces $\Omega^\infty KO\adj \simeq \Omega^\infty L\R\adj$ has been known for a long time and is due to Sullivan.

Notice that the canonical map $L\Z\adj \to L\R\adj$ is an equivalence of $\e_\infty$-ring spectra as well. Thus as a consequence of our corollary we get an equivalence of $\e_\infty$-ring spectra $KO\adj \simeq L\Z\adj$.
\end{Rmk}

\subsection{Applications to assembly maps}\label{assembly}

The equivalence of the two functors $KO\adj$ and $L\adj$ has the following application to the Baum-Connes and Farrell-Jones conjectures.

%Let us denote by $\Orb^\omega$ the full subcategory of the $(2,1)$-category $\Gpd_2$ of groupoids spanned by groupoids with a single object and at most countably many morphisms. We call this the global orbit category. Concretely this is the $(2,1)$-category whose objects are countable discrete groups, whose $1$-morphisms are group homomorphisms and whose $2$-morphisms are conjugations.

By $\Gpd_2^\omega$ we denote the $(2,1)$-category of small groupoids with at most countable many morphisms, here $2$-morphisms are natural transformations. To any such groupoid one can associate a separable $C^*$-algebra called the maximal full groupoid $C^*$-algebra, see \cite[Remark 2.3]{DavisLueck}, \cite[Definition 6]{LNS}, or \cite[3.16]{DellAmbrogio2} together with \cite[section 3]{Joachim2}. As observed in \cite{DavisLueck}, this association is not functorial for all morphisms of groupoids, just for functors that are injective on the set of objects. In \cite[Corollary 8]{LNS} we prove the following proposition in the case of complex $C^*$-algebras. The proof for the real case is verbatim the same. 
\begin{Prop}
There is a functor 
\[\Gpd_2^\omega \to \KK^\R_\infty \]
which on objects sends a groupoid to its real full groupoid $C^*$-algebra.
\end{Prop}

It follows from \cref{KO vs L} that the compositions  $KO\adj, L\adj: \Grp_2^\omega \to  \KK_\infty^\R \to \Sp$ are also equivalent.
Since $\Sp$ is the $\infty$-category associated to a combinatorial model category we obtain that these restrictions can be identified with functors 
\[ KO\adj, L\adj \colon \Gpd^\omega \to \Sp_1 \]
that have the property that they send equivalences of groupoids to equivalences of spectra \cite[Proposition 4.2.4.4]{LurieHTT}. Furthermore it follows that  these functors are related through a zig-zag of natural weak equivalences. It follows from the work of Davis and L\"uck  \cite{DavisLueck} that their associated assembly maps are equivalent. For the further discussion we adopt the terminology from  \cite{DavisLueck}. 

It is well-known that the Baum-Connes assembly map (using the reduced group $C^*$-algebra) factors through the version with the full group $C^*$-algebra.
We obtain for every countable discrete group $G$ a commutative diagram (the commutativity being ensured by the fact that the two \emph{functors} are equivalent)
\[\xymatrix{KO^G_*(\underline{E}G)\adj \ar[rr] \ar[d]_\cong^\tau & & KO_*(C^*(G;\R))\adj \ar[d]^\cong_\tau \\ L\R^G_*(\underline{E}G)\adj \ar[r]^-{\FJ\adj} & L_*(\R G)\adj \ar[r] & L_*(C^*(G;\R))\adj}\]
where the assembly map $\FJ\adj$  is isomorphic to the 2-inverted version of the map which is conjectured to be an isomorphism in the  Farrell-Jones conjecture. This follows from the fact that the equivariant $L$-theory groups of the classifying space for virtually cyclic subgroups of $G$ and the equivariant $L$-theory groups of the classifying space for finite subgroups of $G$ are isomorphic after inverting 2, see \cite[Proposition 2.18]{LueckReich}.

Furthermore the two squares
\[ \xymatrix{L\R^G_*(\underline{E}G)\adj \ar[r] & L_*(\R G)\adj  &
KO_*(C^*(G;\R))\adj \ar[d]^\cong_\tau \ar[r] & KO_*(C^*_r (G;\R))\adj \ar[d]^\cong_\tau \\ 
L^q\Z^G_*(\underline{E}G)\adj \ar[u]^\cong \ar[r] & L^q_*(\Z G)\adj \ar[u] & L_*(C^*(G;\R))\adj \ar[r] & L_*(C^*_r(G;\R))\adj}\]
commute as well where $L^q\Z$ is quadratic $L$-theory of the integers. Thus we can paste these diagrams together to obtain the following theorem.

\begin{Thm}\label{BC und FJ}
Let $G$ be a countable discrete group. Then the following diagram is commutative.
\[\xymatrix@C=1.5cm{ KO^G_*(\underline{E}G)\adj \ar[rr]^-{\BC\adj} \ar[d]_\cong^\tau & & KO_*(C^*_r(G;\R))\adj  \ar[d]^-\cong_\tau \\ L\R^G_*(\underline{E}G)\adj \ar[r]^-{\FJ\adj} & L_*(\R G)\adj \ar[r] & L_*(C^*_r(G;\R))\adj \\ L^q\Z^G_*(\underline{E}G)\adj \ar[u]^\cong \ar[r]_-{\FJ\adj} & L^q_*(\Z G)\adj \ar[u] &}\]
\end{Thm}
Note that the upper horizontal map $\BC\adj$ in this diagram is the 2-inverted version of the map which is conjectured to be an isomorphism in the Baum-Connes conjecture.
We obtain the following direct consequence.
\begin{Cor}
Let $G$ be a countable discrete group. Suppose the real Baum-Connes map is injective after inverting $2$. Then so is the Farrell-Jones map in quadratic $L$-theory for the ring $\Z$.
\end{Cor}
\begin{Rmk}
Recall that the real Baum-Connes assembly map is injective in all degrees after inverting 2 if and only if the complex Baum-Connes assembly map is injective in all degrees after inverting 2, \cite[Corollary 2.13]{Schick}. Thus the above corollary remains true if the real Baum-Connes map is replaced by the complex one provided one deals with all degrees at once. In the above version injectivity is inherited for each single degree separately.
\end{Rmk}

We now formulate the \emph{completion conjecture in $L$-theory}.
\begin{Conj}
The map induced by the completion 
\[ L_*(\R G)\adj \to L_*(C^*_r(G;\R))\adj \]
is an isomorphism.
\end{Conj}
\begin{Rmk}
We observe that the map
\[ L_0(\R G) \to L_0(C^*_r(G;\R)) \]
is in general not an isomorphism. To see this we consider $G = \Z^n$. We use the Shaneson splitting to calculate $L_0(\R \Z^n)$ and see that it is torsion free. Then we observe that \cref{K groups isomorphic to L groups} remains valid in degree $0$ for $R^*$-algebras. We can then use the Baum-Connes conjecture (which is confirmed for free abelian groups) to calculate that 
\[ L_0(C^*_r(\Z^n;\R)) \cong KO_0(C^*_r(\Z^n;\R)) \cong KO_0(B\Z^n)) \]
and notice that since $T^n \simeq B\Z^n$ splits stably we see $2$-torsion in these groups coming from the $2$-torsion of $\pi_*(KO)$ as soon as $n$ is big enough.
Thus the completion conjecture without inverting $2$ is not valid.
\end{Rmk}

We conclude with the observation that from the commutative diagram of \cref{BC und FJ} we see that there is a 3-for-2 property for the following statements:
\begin{enumerate}
\item[(i)] The Baum-Connes conjecture holds after inverting $2$ for the discrete group $G$,
\item[(ii)] The Farrell-Jones conjecture over $\R$ holds after inverting $2$ for the discrete group $G$,
\item[(iii)] The completion conjecture in $L$-theory holds for the discrete group $G$.
\end{enumerate}

%%%%%%%%%%%%%%%%%%%%%%%%%%%%%%%%%%%%%%%%%%%%%%%%%%%%%%%%%%%%%%%%%%%%%%%%%%%%%%%%%%%%%%%%%%%%%%%%%%%%%%%%%%%%%%%%%%%%%%%%%%%%

\section{Integral maps between $KU$ and $L\C$}\label{integral maps}

We have shown that there exists a natural map $\tau\colon k \to \ell$ which induces an equivalence between the periodic versions after inverting $2$. One could hope that there also exists an integral map between the periodic versions which becomes an equivalence after inverting $2$ or a map in the other direction. It turns out that this is not the case.
This is a consequence of the following theorem.

\begin{Thm}\label{connectedness-of-mapping-spaces}
We have that
\[ [L\C,KU] = [KU,L\C] = [\ell\C,KU] = [\ell\C,ku] = 0 \]
i.e. any such map is null homotopic.
\end{Thm}
The proof of this theorem will proceed in several steps. We need a couple of lemmas to get started. 
A first observation is the following general
\begin{Prop}
Let $R$ be a ring spectrum and $M$ be an $R$-module spectrum. If $R$ is admits the structure of an $\HZ$-module then so does $M$.
\end{Prop}
\begin{proof}
We first notice that it is equivalent for a spectrum $X$ to admit the structure of an $\HZ$-module and to be equivalent to the generalized Eilenberg-MacLane spectrum on its homotopy groups.
The module multiplication map and the unit of the ring spectrum give a factorization of the identity of $M$ 
\[\xymatrix{M \ar[r] & R\otimes M \ar[r] & M}\]
which shows that $M$ is a retract of $R\otimes M$. Since $R$ is an $\HZ$-module, so is $R\otimes M$. By the above this implies that $R\otimes M$ is a generalized Eilenberg-MacLane spectrum. So the proposition follows if we show that the category of generalized Eilenberg-MacLane spectra is closed under retracts. It follows from the functoriality of homotopy groups that if $Y$ is a retract of $X$ then $\pi_*(Y)$ is a retract of $\pi_*(X)$ and it is easy to see that the map $Y \to X \to \mathrm{H}\pi_*(X) \to \mathrm{H}\pi_*(Y)$ is an equivalence.
\end{proof}

\begin{Cor}\label{L-theory-splits-2-locally}
For all $A \in \cc$, the spectrum $LA_{(2)}$ admits the structure of an $\HZ$-module.
\end{Cor}
\begin{proof}
First we need to recall that $L\Z$ is an algebra over $\MSO$ due to the Sullivan-Ranicki orientation $\MSO \xrightarrow{\sigma} L\Z$. In particular for any $A \in \cc$ the spectrum $LA$ is a module over $\MSO$ und thus $LA_{(2)}$ is a module over $\MSO_{(2)}$ which admits the structure of an $\HZ$-module, see e.g. \cite[Theorem A]{Williams}.
\end{proof}
\begin{Rmk}
The $\HZ$-module structure on $LA_{(2)}$ is not canonical, but for our purposes it suffices to choose some $\HZ$-module structure for each algebra $A$.
\end{Rmk}

\begin{Cor}\label{not-equivalent}
The spectra $KU$ and $L\C$ are \emph{not} equivalent, although their homotopy groups are naturally isomorphic.
\end{Cor}
\begin{proof}
It is well known that $KU_{(2)}$ does not admit the structure of an $\HZ$ module.
\end{proof}

\begin{Def}
We call a spectrum $X$ \emph{even} if all its odd homotopy groups vanish. We remark that any such spectrum is complex orientable as the obstructions to being complex oriented lie in odd homotopy groups of $X$. 
\end{Def}
\begin{Lemma}\label{even-homotopy}
If $E$ and $F$ are even homotopy ring spectra and $E$ is Landweber exact, then $E\otimes F$ is even. In particular $KU\otimes L\C$ and $KU\otimes \ell\C$ are even.
\end{Lemma}
\begin{proof}
Since $E$ is assumed to be Landweber exact we have
\[ E_*F \simeq \MU_*F \otimes_{\MU_*} E_* \]
and since $\MU_*$ and $E_*$ are even it thus suffices to prove that $\MU_*F$ is even. For this we see that
\[ \MU_* F = F_* \MU \cong F_*(\BU) \]
by the Thom-isomorphism for $F$. But since $F$ is even and $\BU$ has even homology the Atiyah Hirzebruch spectral sequence implies that $F_*(\BU)$ is even.
\end{proof}

\begin{Lemma}
Suppose $R$ is a torsion free commutative ring such that the additive and the multiplicative formal group law are isomorphic. Then $R$ is a $\Q$-algebra.
\end{Lemma}
\begin{proof}
We can formally write down the logarithm of the multiplicative formal group law and see that this forces all primes to act invertibly on $R$.
\end{proof}

\begin{Cor}\label{rational}
The spectrum $KU \otimes \HZ$ is rational.
\end{Cor}
\begin{proof}
This is a classical fact, see \cite{AH} for a more general statement. A nice proof using formal groups goes as follows.
The spectrum $KU\otimes \HZ$ has two complex orientations, one coming from $KU$ and one coming from $\HZ$. Thus on $\pi_*(KU\otimes \HZ)$ the additive and the multiplicative formal group law are isomorphic: This is a general fact, see \cite[pages 3,4]{Rezk2}. Since by \cref{even-homotopy} the spectrum $KU \otimes \HZ$ is even periodic one can shift the coefficients of the formal group law to degree $0$. We then obtain that $\pi_0(KU\otimes \HZ)$ is a ring on which the additive and the multiplicative formal group law are isomorphic and hence is a $\Q$-algebra. Since $\pi_*(KU\otimes \HZ)$ is a module over $\pi_0(KU\otimes \HZ)$ the corollary follows.
\end{proof}

\begin{Lemma}\label{pullback}
Let $\S$ be the sphere spectrum and $p$ be a prime. Then the diagram
\[\xymatrix{\S \ar[r] \ar[d] & \S[\tfrac{1}{p}] \ar[d] \\ \S_{(p)} \ar[r] & \S_\Q}\]
is a pullback diagram of spectra.
\end{Lemma}
\begin{proof}
This follows from the equivalence of the horizontal cofibers which are the Moore spectra $M(\Z[\tfrac{1}{p}]/\Z) \simeq M(\Q / \Z_{(p)})$.
\end{proof}

We apply this observation as follows.
\begin{Lemma}\label{equivalence}
The canonical map 
\[ \xymatrix{L\C \otimes KU \ar[r] & L\C\otimes KU \otimes \mathbb{S}\adj = (L\C\otimes KU)\adj}\]
is an equivalence of $L\C \otimes KU$-modules.
\end{Lemma}
\begin{proof}
The map is clearly a module map. So it suffices to argue that it is an equivalence of spectra.
For this we consider the pullback diagram
\[\xymatrix{L\C\otimes KU \ar[r] \ar[d] & (L\C \otimes KU)\adj \ar[d] \\ (L\C \otimes KU)_{(2)} \ar[r] & (L\C\otimes KU)_\Q}\]
wich is obtained by smashing the pullback diagram of \cref{pullback} with the spectrum $KU \otimes \HZ$. Since pullbacks are pushouts, smashing a pullback diagram with a spectrum gives again a pullback diagram.
Now we observe that $(L\C \otimes KU)_{(2)}$ is an $\HZ\otimes KU$ module since $L\C_{(2)}$ is an $\HZ$-module. By \cref{rational} the spectrum $\HZ\otimes KU$ is rational. Hence also all modules over this spectrum are rational. But this implies that in the above pullback diagram the lower horizontal arrow is an equivalence, thus also the upper horizontal one is.
\end{proof}
\begin{Rmk}
The same is true if we replace $L\C$ by $\ell\C$.
\end{Rmk}

\begin{Cor}
There is an equivalence of mapping spaces
\begin{enumerate}
\item[(1)] $\Map(L\C,KU) \simeq \Map(L\C\adj,KU)$, and
\item[(2)] $\Map(KU,L\C) \simeq \Map(KU\adj,L\C)$, and
\item[(3)] $\Map(\ell\C,ku) \simeq \Map(\ell\C,KU) \simeq \Map(\ell\C\adj,KU)$.
\end{enumerate}
\end{Cor}
\begin{proof}
\cref{equivalence} implies that $L\C \otimes KU \simeq L\C\adj \otimes KU$ as $KU$-modules and $L\C \otimes KU \simeq L\C\otimes KU\adj$ as $L\C$-modules. Thus we obtain
\begin{align*}
\Map(L\C,KU) & \simeq \Map_{KU}(L\C\otimes KU,KU) \\
	& \simeq \Map_{KU}(L\C\adj\otimes KU,KU) \\ & \simeq \Map(L\C\adj,KU)
\end{align*}
Statement $(2)$ follows similarly and statement $(3)$ from the fact that \cref{equivalence} is true for $\ell\C$ instead of $L\C$ and the universal property of connective covers.
\end{proof}

\begin{Prop}\label{anderson duality sequences}
There are short exact sequences
\[\xymatrix@R=.5cm@C=0.5cm{0 \ar[r] & \Ext^1_\Z(KU_{-1}(L\C\adj),\Z) \ar[r] & KU^0(L\C\adj) \ar[r] & \Hom_\Z(KU_0(L\C\adj),\Z) \ar[r] &0\\
	0 \ar[r] & \Ext_\Z^1(KU_{-1}(\ell\C\adj),\Z) \ar[r] & KU^0(\ell\C\adj) \ar[r] & \Hom_\Z(KU_0(\ell\C\adj),\Z) \ar[r] & 0 \\
	0 \ar[r] & \Ext_\Z^1(L\C_{-1}(KU\adj),\Z) \ar[r] & L\C^0(KU\adj) \ar[r] & \Hom_\Z(L\C_0(KU\adj),\Z) \ar[r] & 0 \\
	}\]
\end{Prop}
\begin{proof}
The exact sequences follow from the general UCT sequence relating a spectrum and its Anderson dual, see \cite{anderson}, using that both $KU$ and $L\C$ are Anderson self-dual (see \cite[below Prop. 2.2]{Stojanoska}). 
\end{proof}

\begin{proof}[Proof of \cref{connectedness-of-mapping-spaces}]
The Ext-terms of \cref{anderson duality sequences} vanish due to \cref{even-homotopy}, and certainly \[KU_0(L\C\adj) \cong L\C_0(KU\adj) \text{ and } KU_0(\ell\C\adj) \] are $\Z\adj$-modules. For all $\Z\adj$-modules $M$ we have that $\Hom(M,\Z) = 0$. Thus the theorem follows.
\end{proof}

%%%%%%%%%%%%%%%%%%%%%%%%%%%%%%%%%%%%%%%%%%%%
\appendix \label{appendix}
%%%%%%%%%%%%%%%%%%%%%%%%%%%%%%%%%%%%%%%%%%%%

%%%%%%%%%%%%%%%%%%%%%%%%%%%%%%%%%%%%%%%%%%%%
\section{$L$-theory of non-unital rings}\label{non unital L-theory}
%%%%%%%%%%%%%%%%%%%%%%%%%%%%%%%%%%%%%%%%%%%%

The goal of this section is to prove that the $L$-theory functor for non-unital $C^*$-algebras inherits a lax symmetric monoidal structure from the lax 
symmetric monoidal structure on the $L$-theory functor for unital, involutive rings, see \cref{L theory non unital monoidal}.  We will prove such a result in a more general framework. 

First let $\c$ be an $\infty$-category which has an initial object, finite products and split pullbacks. By split pullback we mean a pullback in which one of the morphisms admits a section. Since we will need these limits repeatedly we introduce some terminology: we consider the collection $\calK$ of simplicial sets consisting of all finite discrete simplicial sets and the nerve of the category depicted as 
$$
\xymatrix@R=.3cm@C=1.2cm{
 & \bullet \ar[dd]<-2pt>_p  & \\ & & ps = \id\\
 \bullet\ar[r] & \bullet \ar[uu]<-2pt>_s & &
} 
$$
which indexes split pullbacks.

In this situation, where $\c$ admits $\calK$-shaped limits and an initial object $\varnothing$, we can form a new $\infty$-category $\c_{/\emptyset}$ consisting of objects of $\c$ together with a morphism to the initial object $\emptyset$. There is a canonical functor
$R: \c \to \c_{/\emptyset}$ which is right adjoint to the canonical forgetful functor $\c_{/\emptyset} \to \c$ and which sends $c \in \c$ to $(c \times \emptyset \to \emptyset) \in \c_{/\emptyset}$. For the next statement, recall that an $\infty$-category is called pointed if it admits an object which is both initial and terminal. 

\begin{Prop}\label{bladddd}
Suppose $\Dd$ is an $\infty$-category which admits $\calK$-shaped limits and is pointed. Then the functor
\[ R^*: \xymatrix{\Fun^\calK(\c_{/\emptyset},\Dd) \ar[r]^-\simeq & \Fun^\calK(\c,\Dd)}\]
is an equivalence of $\infty$-categories.
\end{Prop}
\begin{proof}
Consider the category $I := \Nfin^{\leq 1}$ of finite pointed sets of cardinality less or equal to 1. An $I$-shaped diagram in $\c$ is given by a morphism $c_1 \to c_0$ with a chosen split. We consider the full subcategory $\c_{/\emptyset} \subseteq \Fun(I,\c)$ consisting of those morphisms $c_1 \to c_0$ where $c_0$ is initial in $\c$. Since the section is essentially unique in this case, it follows that this $\infty$-category is equivalent to the usual slice category considered above. The inclusion $\c_{/\emptyset} \subseteq \Fun(I,\c)$ admits a right adjoint given on objects as  
$$
R: \Fun(I,\c) \to \c_{/\emptyset} \qquad (c_1 \to c_0) \mapsto (c_1 \times_{c_0} \emptyset \to \emptyset) .
$$
The functor $R$ exhibits $\c_{/\emptyset}$ as a colocalization of $\Fun(I,\c)$. It is in particular a Dwyer-Kan localization at those morphisms in $\Fun(I,\c)$ which get mapped to equivalences by $R$. Assume now that we have another $\infty$-category $\Dd$ which also admits an initial object and $\calK$-shaped limits  and a functor $F: \c \to \Dd$. By postcomposition we obtain a functor $F_*: \Fun(I,\c) \to \Fun(I,\Dd)$. If $F$ preserves $\calK$-shaped limits then we claim that 
the functor $F_*: \Fun(I,\c) \to \Fun(I,\Dd)$ descends to a functor $F': \c_{/\emptyset} \to \Dd_{/\emptyset}$.
To verify this it suffices to check that $F_*$ sends local equivalences to local equivalences, which precisely follows from the fact that the functor $F$ preserves $\calK$-shaped limits. By the same reasoning one can also check that $F'$ again has the property of preserving $\calK$-shaped limits.

Thus we get a commutative diagram
$$
\xymatrix{
\c \ar[r]^F\ar[d] &\Dd\ar[d] \\
\Fun(I,\c)  \ar[r]^{F_*}\ar[d] & \Fun(I,\Dd)\ar[d] \\
\c_{/\emptyset} \ar[r]^{F'} & \Dd_{/ \emptyset} .  
}
$$
Here the left upper vertical map is the right adjoint to the 
 functor  $\Fun(I,\c) \to \c$ given by evaluation at $\langle 1 \rangle \in I$. This adjoint is given on objects by
$$\c \to \Fun(I,\c) \qquad c \mapsto (c \times c \xrightarrow{\mathrm{pr}_1} c).$$
Similar for the right upper vertical map. The right vertical composition is the cofree functor, in particular it is an equivalence if $\Dd$ is pointed.
We thus obtain a functor
\[\xymatrix{ \Fun^\calK(\c,\Dd) \ar[r] & \Fun^\calK(\c_{/\emptyset}, \Dd_{/\emptyset}) \ar[r]^\simeq &  \Fun^\calK(\c_{/\emptyset}, \Dd). }\]
This comes by construction as an inverse to the functor $R^*$. 
\end{proof}

\begin{Rmk} 
Assume we are given a functor $F: \c \to \Dd$ that preserves $\calK$-shaped limits. Then the proof shows that the extension $\overline{F}: \c_{/\emptyset} \to \Dd$ is informally given by sending $(c \to \emptyset)$ to $(F(c)\times_{F(\emptyset)} \emptyset \to \emptyset)$.
\end{Rmk}

Again consider an $\infty$-category $\c$ that admits $\calK$-shaped limits and an initial object $\emptyset \in \c$. Suppose furthermore that $\c$ is equipped with a symmetric monoidal structure which has the property that the tensor 
bifunctor preserves $\calK$-shaped limits in each variable separately. 
\begin{Prop}\label{blaeeeee}
There exists a symmetric monoidal structure on the slice category $\c_{/\emptyset}$ such that the canonical functor $\c \to \c_{/\emptyset}$ admits a symmetric monoidal refinement.
\end{Prop}
\begin{proof}
We will prove the dual statement, which reads as follows: let $\c$ be a symmetric monoidal $\infty$-category $\c$ which admits $\calK$-shaped colimits and such that the tensor product preserves $\calK$-shaped colimits. If $\c$ admits a terminal object $*$ then $\c_{*} := \c_{*/}$  admits a symmetric monoidal structure such that the functor $\c \to \c_*$ which adds a disjoint basepoint can be refined to a symmetric monoidal functor. 

Before we prove this statement let us comment on how it implies the proposition. This follows from the general fact that the opposite of a symmetric monoidal $\infty$-category admits also the structure of a symmetric monoidal $\infty$-category. This is done by straightening the fibration $\c^\otimes \to \Nfin$, postcomposing with the opposite functor, and then unstraightening again. This construction is functorial in symmetric monoidal functors. For a more detailed discussion see \cite{Knudsen, Barwicketal}.

Now to the proof of the above statement. 
Recall the $\infty$-category  $I := \Nfin^{\leq 1}$ considered in the proof of \cref{bladddd} . This admits a symmetric monoidal structure given by smash product. Thus the functor category $\Fun(I,\c)$ admits a refinement to an $\infty$-operad  $\Fun(I, \c)^\otimes$ given by Day convolution. For details see \cite{Glasman} and more specifically \cite[Proposition 3.3]{Nikolaus}. We claim that in this situation the Day convolution is symmetric monoidal and the canonical functor $\Fun(I, \c)^\otimes \to \c^\otimes$ admits an operadic left adjoint which is symmetric monoidal (for the terminology operadic left adjoint we refer to \cite[Definition 2.9]{Nikolaus}). The proof is essentially the same proof as in \cite[Lemma 2.5]{Glasman} except that we only have to make sure that the all the colimits in the left Kan extension describing the tensor product exists in $\c$. To see this first notice that  the Day convolution tensor product is, if it exists, pointwise given by the pushout-product in the functor category. Since all morphisms have chosen splits one easily checks that this  only requires split pushouts in $\c$ which exist by assumption. 

As a next step we claim that the symmetric monoidal structure on $\Fun(I,\c)$ descends to a symmetric monoidal structure along the localization $\Fun(I,\c) \to \c_*$. To see this we need to prove that local equivalences are preserved by tensoring with any object of $\Fun(I,\c)$ (see \cite[Example 2.12.]{Nikolaus} for a precise statement of this well known criterion). Again this can now be directly verified using the pushout product axiom and the fact that the tensor product of $\c$ preserves $\calK$-shaped colimits.
\end{proof}

\begin{Rmk}
\cref{blaeeeee} also admits a different proof using the methods developed in \cite{GGN}. To do that one has to consider the tensor product on the $\infty$-category $\Cat_\infty^\calK$ consisting of all $\infty$-categories that admit $\calK$-shaped colimits and functors that preserve those. Then it follows from Proposition \ref{bladddd}  that the construction $\c \mapsto \c_*$ is a smashing localization of $\Cat_\infty^\calK$.  
\end{Rmk}

We keep the assumption that $\c$ is a symmetric monoidal $\infty$-category that admits $\calK$-shaped limits, an initial object $\emptyset \in \c$ and such that the tensor 
bifunctor preserves $\calK$-shaped limits in each variable separately.
Now assume that $\c_{/\emptyset}$ is equipped with a symmetric monoidal structure and that the functor $\c \to \c_{/\emptyset}$ admits a refinement to a symmetric monoidal functor. This can always be done by \cref{blaeeeee}. But it could also be done by other means (as will be the case in our application) and the next statement will be true in such a potentially more general situation. 
Recall that a presentably symmetric monoidal $\infty$-category is a symmetric monoidal $\infty$-category which is presentable as an $\infty$-category and such that the tensor bifunctor preserves colimits separately in both variables. 

\begin{Prop}\label{symmetric structure on slice}
 Suppose $\Dd$ is a presentably symmetric monoidal $\infty$-category which is pointed. Then the restriction along the above functor induces an equivalence
\[ \xymatrix{\Fun_\lax^\calK(\c_{/\emptyset},\Dd) \ar[r]^-\simeq & \Fun_\lax^\calK(\c,\Dd)}\]
where the superscript $\calK$ denotes functors that preserve $\calK$-shaped limits.
\end{Prop}
\begin{proof}

The functor categories $\Fun(\c,\Dd)$ and $\Fun(\c_{/\emptyset},\Dd)$ are equipped with symmetric monoidal structures $\Fun(\c,\Dd)^\otimes$ and $\Fun(\c_{/\emptyset},\Dd)^\otimes$ via the Day convolution, just as in \cite[Proposition 3.3]{Nikolaus}.
By functoriality of the Day convolution the functor $\c \to \c_{/\emptyset}$ induces a lax symmetric monoidal functor
\[\xymatrix{\Fun(\c_{/\emptyset},\Dd) \ar[r] & \Fun(\c,\Dd)}\]
i.e. a map of $\infty$-operads $ \Fun(\c_{/\emptyset},\Dd)^\otimes \to \Fun(\c,\Dd)^\otimes$. 
This map fits into a commutative square of lax symmetric monoidal functors
\begin{equation}\label{diagramp}
\xymatrix{\Fun(\c_{/\emptyset},\Dd) \ar[r]  & \Fun(\c,\Dd)  \\ \Fun^\calK(\c_{/\emptyset},\Dd) \ar[r]\ar[u] & \Fun^\calK(\c,\Dd)\ar[u] }
\end{equation}
where the vertical maps are full inclusions and the bottom map is defined by restriction of the upper one. This makes sense since $\c \to \c_{/\emptyset}$ preserves $\calK$-shaped limits. 

We claim that all maps in this square admit operadic left adjoints which are strong symmetric monoidal. For the upper horizontal arrow this is  \cite[Corollary 3.8]{Nikolaus}. 
The inclusion $\Fun^\calK(\c,\Dd) \subseteq \Fun(\c,\Dd)$ of $\infty$-categories admits a left adjoint and thus is a reflective subcategory. To show the existence of a symmetric monoidal left adjoint we invoke the criterion given in \cite[Example 2.12(3)]{Nikolaus}: 
we need to argue that if $G: \c \to \Dd$ preserves $\calK$-shaped limits, then for every other functor $F: \c \to \Dd$, the internal hom from $F$ to $G$ in $\Fun(\c,\Dd)^\otimes$ preserves $\calK$-shaped limits as well. For the internal hom functor we have the formula
\[ \underline{\map}_{\Fun(\c,\Dd)}(F,G)(X) = \int_{Y\in \c } \underline{\map}_\Dd\Big(F(Y),G(X\otimes Y)\Big),\]
this is proven  in \cite[Proposition 3.11]{Nikolaus}. 
The tensor bifunctor of $\c$ preserves by assumption $\calK$-shaped limits in each variable separately, the functor $G$ preserves $\calK$-shaped limits as well as the internal hom and the end. Thus the internal hom does preserve $\calK$-shaped limits in $X$.

Hence we have shown that in Diagram (\ref{diagramp}) above the right vertical functor admits a symmetric monoidal left adjoint.
The same argument works for the left vertical functor. Finally the left adjoint of the lower horizontal morphism is then obtained as the localization of the left adjoint to the upper one using the symmetric monoidal universal property of the localization functor $\Fun(\c,\Dd) \to \Fun^\calK(\c,\Dd)$.

Thus the lax symmetric monoidal functor 
\[ i^*: \Fun^\calK(\c_{/\emptyset},\Dd) \to \Fun^\calK(\c,\Dd) \]
is an  equivalence of underlying $\infty$-categories and admits a symmetric monoidal left adjoint. It follows that the left adjoint is an equivalence of symmetric monoidal $\infty$-categories which follows from the fact that it is an underlying equivalence and symmetric monoidal. Thus the operadic right adjoint $i^*$ is an equivalence of symmetric monoidal  $\infty$-categories as well and in particular a symmetric monoidal functor (not only lax). 
As a result we obtain an equivalence 
\[\Alg_{\e_\infty}(i^*): \xymatrix{ \Alg_{\e_\infty}(\Fun^\calK(\c_{/\emptyset},\Dd)) \ar[r]^-{\simeq} & \Alg_{\e_\infty}(\Fun^\calK(\c,\Dd)) }.\]
Since the inclusions $\Fun^\calK(\c,\Dd) \subseteq \Fun(\c,\Dd)$ (and likewise for $\c_{/\emptyset}$) are operadically fully faithful we deduce from \cite{Glasman} that we have equivalences 
\[ \Alg_{\e_\infty}(\Fun^\calK(\c,\Dd)) \simeq \Fun_\lax^\calK(\c,\Dd) \]
and 
\[ \Alg_{\e_\infty}(\Fun^\calK(\c_{/\emptyset},\Dd)) \simeq \Fun_\lax^\calK(\c_{/\emptyset},\Dd). \]
Putting everything together proves the proposition.
\end{proof}

\begin{Rmk}
One can show that the last proposition is true for more general $\Dd$, namely for $\infty$-operads which admit $\calK$-shaped operadic limits and are operadically pointed. The last proposition never used anything about the symmetric monoidal structure on $\c_{/\emptyset}$ other than the existence of a symmetric monoidal refinement of the canonical map $\c \to \c_{/\emptyset}$. This shows that the symmetric monoidal structure on $\c_{/\emptyset}$ is essentially unique in the appropriate sense.
\end{Rmk}

We now want to apply this abstract criterion to the case $\c = \N\ccu$. Then we have an equivalence  $\N\cc\simeq \c_{/\emptyset}$ given by sending a non-unital $C^*$-algebra $A$ to its unitalization $A^+ \to \C$. Under this equivalence the symmetric monoidal structure $\c_{/\emptyset}$ corresponds to the usual tensor product of non-unital $C^*$-algebras. This follows either by directly comparing the definitions (these are ordinary categories) or using the uniqueness assertion made in the last remark. Now we recall the results of \cite{Laures2} and \cite{Laures} saying that $L$-theory for unital rings with involution admits a symmetric monoidal structure.
There it is shown that the functor
\[ \xymatrix{\ringsinv \ar[r]^-{L} & \Sp^\Sigma_1}\] has a lax symmetric monoidal refinement, where $\Sp_1^\Sigma$ denotes the symmetric monoidal model category of symmetric spectra. It thus follows that the induced functor
\[ \xymatrix{\N\ccu \ar[r] & \N\ringsinv \ar[r]^-{L} & \Sp}\]
of $\infty$-categories also admits a lax symmetric monoidal refinement because the functor 
\[\xymatrix{\N\ccu \ar[r] & \N\ringsinv }\] 
is lax symmetric monoidal as well.

\begin{Prop}\label{L theory non unital monoidal}
The  $L$-theory functor $L \in \Fun_\lax^\calK(\N\ccu,\Sp)$ admits an essentially unique refinement to a functor $L \in \Fun_\lax^\calK(\N\cc,\Sp)$.
\end{Prop}
\begin{proof}
This follows immediately from the observation made above that for $\c = \N\ccu$ the slice category $\c_{/\emptyset}$ is equivalent to $\N\cc$ and \cref{symmetric structure on slice}, where we use that the category $\Sp$ is a presentably symmetric monoidal $\infty$-category.
\end{proof}

%\end{appendix}

\bibliographystyle{amsalpha}
\bibliography{mybib}

\providecommand{\bysame}{\leavevmode\hbox to3em{\hrulefill}\thinspace}
\providecommand{\MR}{\relax\ifhmode\unskip\space\fi MR }
\providecommand{\MRhref}[2]{%
  \href{http://www.ams.org/mathscinet-getitem?mr=#1}{#2}
}
\providecommand{\href}[2]{#2}
\begin{thebibliography}{MEKD11}

\bibitem[AH68]{AH}
D.~W. Anderson and L.~Hodgkin, \emph{The {$K$}-theory of
  {E}ilenberg-{M}ac{L}ane complexes}, Topology \textbf{7} (1968), 317--329.
  \MR{0231369}

\bibitem[And70]{anderson}
D.~W. Anderson, \emph{Universal {C}oefficient {T}heorems for {K}-theory}, MIT
  Department of Mathematics (1970), no.~23.

\bibitem[BGN14]{Barwicketal}
C.~{Barwick}, S.~{Glasman}, and D.~{Nardin}, \emph{{Dualizing cartesian and
  cocartesian fibrations}}, arXiv:1409.2165 (2014).

\bibitem[BJM17]{BJM}
Ilan Barnea, Michael Joachim, and Snigdhayan Mahanta, \emph{Model structure on
  projective systems of {$C^*$}-algebras and bivariant homology theories}, New
  York J. Math. \textbf{23} (2017), 383--439. \MR{3649664}

\bibitem[Bla98]{Blackadar}
B.~Blackadar, \emph{{$K$}-theory for operator algebras}, second ed.,
  Mathematical Sciences Research Institute Publications, vol.~5, Cambridge
  University Press, Cambridge, 1998. \MR{1656031 (99g:46104)}

\bibitem[Bla06]{Blackadar2}
\bysame, \emph{Operator algebras}, Encyclopaedia of Mathematical Sciences, vol.
  122, Springer-Verlag, Berlin, 2006, Theory of $C{^{*}}$-algebras and von
  Neumann algebras, Operator Algebras and Non-commutative Geometry, III.
  \MR{2188261 (2006k:46082)}

\bibitem[Bou90]{Bousfield}
A.~K. Bousfield, \emph{A {C}lassification of {K}-local spectra}, Journal of
  Pure and Applied Algebra \textbf{66} (1990), 121--163.

\bibitem[Cis10]{Cisinski}
D.-C. Cisinski, \emph{Invariance de la {$K$}-th{\'e}orie par {\'e}quivalences
  d{\'e}riv{\'e}es}, J. K-Theory \textbf{6} (2010), no.~3, 505--546.
  \MR{2746284 (2012h:19006)}

\bibitem[CLM16]{CLM}
D.~Crowley, W.~L{{\"u}}ck, and T.~Macko, \emph{Introduction to {S}urgery
  {T}heory}, available at http://131.220.77.52/lueck/ (2016).

\bibitem[Del12]{DellAmbrogio2}
I.~Dell'Ambrogio, \emph{The unitary symmetric monoidal model category of small
  {$\rm C^*$}-categories}, Homology Homotopy Appl. \textbf{14} (2012), no.~2,
  101--127. \MR{3007088}

\bibitem[DK80]{Dwyer}
W.~G. Dwyer and D.~M. Kan, \emph{Simplicial localizations of categories}, J.
  Pure Appl. Algebra \textbf{17} (1980), no.~3, 267--284. \MR{579087
  (81h:55018)}

\bibitem[DL98]{DavisLueck}
J.~Davis and W.~L\"uck, \emph{Spaces over a {C}ategory, {A}ssembly {M}aps, and
  {I}somorphism {C}onjectures in {K}- and {L}-{T}heory}, K-Theory \textbf{15}
  (1998), 201--251.

\bibitem[GGN15]{GGN}
D.~Gepner, M.~Groth, and T.~Nikolaus, \emph{Universality of multiplicative
  infinite loop space machines}, Algebr. Geom. Topol. \textbf{15} (2015),
  3107--3153.

\bibitem[Gla16]{Glasman}
Saul Glasman, \emph{Day convolution for {$\infty$}-categories}, Math. Res.
  Lett. \textbf{23} (2016), no.~5, 1369--1385. \MR{3601070}

\bibitem[Hig87]{Higson87}
N.~Higson, \emph{A characterization of {$KK$}-theory}, Pacific J. Math.
  \textbf{126} (1987), no.~2, 253--276. \MR{869779 (88a:46083)}

\bibitem[Hig88]{Higson88}
\bysame, \emph{Algebraic {$K$}-theory of stable {$C^*$}-algebras}, Adv. in
  Math. \textbf{67} (1988), no.~1, 140. \MR{922140 (89g:46110)}

\bibitem[Hin16]{Hinich}
V.~Hinich, \emph{Dwyer-{K}an localization revisited}, Homology Homotopy Appl.
  \textbf{18} (2016), no.~1, 27--48. \MR{3460765}

\bibitem[Joa03]{Joachim2}
M.~Joachim, \emph{{$K$}-homology of {$C^\ast$}-categories and symmetric spectra
  representing {$K$}-homology}, Math. Ann. \textbf{327} (2003), no.~4,
  641--670. \MR{2023312}

\bibitem[Joa04]{Joachim}
\bysame, \emph{Higher coherences for equivariant {$K$}-theory}, Structured ring
  spectra, London Math. Soc. Lecture Note Ser., vol. 315, Cambridge Univ.
  Press, Cambridge, 2004, pp.~87--114. \MR{2122155 (2006j:19004)}

\bibitem[Joy08]{Joyal}
Andr{\'e} Joyal, \emph{Notes on quasi-categories}, preprint, 2008.

\bibitem[Kar80]{Karoubi}
M.~Karoubi, \emph{Th{\'e}orie de {Q}uillen et homologie du groupe orthogonal},
  Ann. of Math. (2) \textbf{112} (1980), no.~2, 207--257. \MR{592291
  (82h:18011)}

\bibitem[Kas75]{KasparovI}
G.~G. Kasparov, \emph{Topological invariants of elliptic operators {I},
  {K}-{H}omology}, Izv. Akad. Nauk SSSR \textbf{39} (1975), 796--838.

\bibitem[Kas80]{KasparovIV}
\bysame, \emph{The operator {$K$}-functor and extensions of {$C^{\ast}
  $}-algebras}, Izv. Akad. Nauk SSSR Ser. Mat. \textbf{44} (1980), no.~3,
  571--636, 719. \MR{582160}

\bibitem[Kas88]{KasparovIII}
\bysame, \emph{Equivariant {K}{K}-theory and the {N}ovikov conjecture},
  Inventiones mathematicae \textbf{91} (1988), 147--201.

\bibitem[Kas95]{KasparovII}
\bysame, \emph{{$K$}-theory, group {$C^*$}-algebras, and higher signatures
  (conspectus)}, Novikov conjectures, index theorems and rigidity, {V}ol.\ 1
  ({O}berwolfach, 1993), London Math. Soc. Lecture Note Ser., vol. 226,
  Cambridge Univ. Press, Cambridge, 1995, pp.~101--146. \MR{1388299}

\bibitem[{Knu}16]{Knudsen}
B.~{Knudsen}, \emph{{Higher enveloping algebras}}, arXiv:1605.01391, May 2016.

\bibitem[LM13]{Laures2}
G.~Laures and J.~E. McClure, \emph{Commutativity properties of {Q}uinn
  spectra}, arXiv:1304.4759 (2013).

\bibitem[LM14]{Laures}
\bysame, \emph{Multiplicative properties of {Q}uinn spectra}, Forum Math.
  \textbf{26} (2014), no.~4, 1117--1185. \MR{3228927}

\bibitem[LNS17]{LNS}
Markus Land, Thomas Nikolaus, and Karol Szumi{\l}o, \emph{Localization of
  cofibration categories and groupoid {C{$^*$}}--algebras}, Algebr. Geom.
  Topol. \textbf{17} (2017), no.~5, 3007--3020. \MR{3704250}

\bibitem[LR05]{LueckReich}
W.~L\"uck and H.~Reich, \emph{The {B}aum-{C}onnes and {F}arrell-{J}ones
  {C}onjectures in {K}- and {L}-{T}heory}, Handbook of K-Theory \textbf{2}
  (2005), 703--842.

\bibitem[Lur09]{LurieHTT}
J.~Lurie, \emph{{H}igher {T}opos {T}heory}, Annals of Mathematics Studies, vol.
  170, Princeton University Press, Princeton, NJ, 2009. \MR{2522659
  (2010j:18001)}

\bibitem[Lur11]{Lurie}
\bysame, \emph{Lecture on {A}lgebraic {L}-{T}heory and {S}urgery},
  http://www.math.harvard.edu/~lurie/287x.html, 2011.

\bibitem[Lur14]{LurieHA}
\bysame, \emph{{H}igher {A}lgebra}, http://www.math.harvard.edu/~lurie/, 2014.

\bibitem[MEKD11]{Meyer2}
R.~Meyer, H.~Emerson, T.~Kandelaki, and I.~Dell'Ambrogio, \emph{An equivariant
  lefschetz fixed-point formula for correspondences}, Arxiv:1104.3441v1 (2011).

\bibitem[Mil71]{Milnor}
J.~Milnor, \emph{Introduction to algebraic {$K$}-theory}, Princeton University
  Press, Princeton, N.J.; University of Tokyo Press, Tokyo, 1971, Annals of
  Mathematics Studies, No. 72. \MR{0349811 (50 \#2304)}

\bibitem[Mil98]{Miller}
J.~G. Miller, \emph{Signature operators and surgery groups over
  {$C^*$}-algebras}, $K$-Theory \textbf{13} (1998), no.~4, 363--402.
  \MR{1615684 (2000d:19003)}

\bibitem[MM79]{MM}
I.~Madsen and R.~J. Milgram, \emph{The classifying spaces for surgery and
  cobordism of manifolds}, Annals of Mathematics Studies, vol.~92, Princeton
  University Press, Princeton, N.J.; University of Tokyo Press, Tokyo, 1979.
  \MR{548575 (81b:57014)}

\bibitem[Nik16]{Nikolaus}
T.~Nikolaus, \emph{Stable {$\infty$}-operads and the multiplicative {Y}oneda
  embedding}, ArXiv:1608.02901 (2016).

\bibitem[Ran73a]{Ranicki1}
A.~A. Ranicki, \emph{Algebraic {$L$}-theory. {I}. {F}oundations}, Proc. London
  Math. Soc. (3) \textbf{27} (1973), 101--125. \MR{0414661 (54 \#2760a)}

\bibitem[Ran73b]{Ranicki2}
\bysame, \emph{Algebraic {$L$}-theory. {II}. {L}aurent extensions}, Proc.
  London Math. Soc. (3) \textbf{27} (1973), 126--158. \MR{0414662 (54 \#2760b)}

\bibitem[Ran73c]{Ranicki3}
\bysame, \emph{Algebraic {$L$}-theory. {III}. {T}wisted {L}aurent extensions},
  Algebraic {K}-theory, {III}: {H}ermitian {K}-theory and geometric application
  ({P}roc. {C}onf. {S}eattle {R}es. {C}enter, {B}attelle {M}emorial {I}nst.,
  1972), Springer, Berlin, 1973, pp.~412--463. Lecture Notes in Mathematics,
  Vol. 343. \MR{0414663 (54 \#2760c)}

\bibitem[Ran74]{Ranicki4}
\bysame, \emph{Algebraic {$L$}-theory. {IV}. {P}olynomial extension rings},
  Comment. Math. Helv. \textbf{49} (1974), 137--167. \MR{0414664 (54 \#2760d)}

\bibitem[Ran79]{RanickiTSO}
\bysame, \emph{The total surgery obstruction}, Algebraic topology, {A}arhus
  1978 ({P}roc. {S}ympos., {U}niv. {A}arhus, {A}arhus, 1978), Lecture Notes in
  Math., vol. 763, Springer, Berlin, 1979, pp.~275--316. \MR{561227
  (81e:57034)}

\bibitem[Ran81]{Ranicki}
\bysame, \emph{Exact sequences in the algebraic theory of surgery},
  Mathematical Notes, vol.~26, Princeton University Press, Princeton, N.J.;
  University of Tokyo Press, Tokyo, 1981. \MR{620795 (82h:57027)}

\bibitem[Ran92]{RanickiBlue}
\bysame, \emph{Algebraic {$L$}-theory and topological manifolds}, Cambridge
  Tracts in Mathematics, vol. 102, Cambridge University Press, Cambridge, 1992.
  \MR{1211640 (94i:57051)}

\bibitem[Rez07]{Rezk2}
C.~Rezk, \emph{{Supplementary notes for Math 512}}, available at
  https://faculty.math.illinois.edu/~rezk/512-spr2001-notes.pdf (2007).

\bibitem[RLL00]{Rordam}
M.~R{\o}rdam, F.~Larsen, and N.~Laustsen, \emph{An {I}ntroduction to
  {$K$}-theory for {$C^*$}-algebras}, London Mathematical Society Student
  Texts, vol.~49, Cambridge University Press, Cambridge, 2000. \MR{1783408
  (2001g:46001)}

\bibitem[Ros95]{Rosenberg}
J.~Rosenberg, \emph{Analytic {N}ovikov for topologists}, Novikov conjectures,
  index theorems and rigidity, {V}ol.\ 1 ({O}berwolfach, 1993), London Math.
  Soc. Lecture Note Ser., vol. 226, Cambridge Univ. Press, Cambridge, 1995,
  pp.~338--372. \MR{1388305 (97b:58138)}

\bibitem[Ros16]{Rosenberg2}
Jonathan Rosenberg, \emph{Novikov's conjecture}, Open problems in mathematics,
  Springer, [Cham], 2016, pp.~377--402. \MR{3526942}

\bibitem[Sch04]{Schick}
T.~Schick, \emph{Real vs complex {K}-theory using {K}asparov's bivariant
  {KK}-theory}, Algebraic and Geometric Topology \textbf{4} (2004), 333--346.

\bibitem[SH14]{Stojanoska}
V.~Stojanoska and D.~Heard, \emph{K-theory, reality, and duality}, Journal of
  K-Theory \textbf{14} (2014), no.~3, 526--555.

\bibitem[Szu14]{Szumilo}
K.~Szumi{\l}o, \emph{Two models for the homotopy theory of cocomplete homotopy
  theories}, no.~Arxiv:1411.0303v1.

\bibitem[Tak02]{Takesaki}
M.~Takesaki, \emph{Theory of operator algebras. {I}}, Encyclopaedia of
  Mathematical Sciences, vol. 124, Springer-Verlag, Berlin, 2002, Reprint of
  the first (1979) edition, Operator Algebras and Non-commutative Geometry, 5.
  \MR{1873025 (2002m:46083)}

\bibitem[TW79]{Williams}
L.~Taylor and B.~Williams, \emph{Surgery spaces: formulae and structure},
  Algebraic topology, {W}aterloo, 1978 ({P}roc. {C}onf., {U}niv. {W}aterloo,
  {W}aterloo, {O}nt., 1978), Lecture Notes in Math., vol. 741, Springer,
  Berlin, 1979, pp.~170--195. \MR{557167 (81k:57034)}

\bibitem[Uuy13]{Uuye}
O.~Uuye, \emph{Homotopical algebra for {$\rm C^*$}-algebras}, J. Noncommut.
  Geom. \textbf{7} (2013), no.~4, 981--1006. \MR{3148615}

\bibitem[Wal99]{Wall}
C.~T.~C. Wall, \emph{Surgery on compact manifolds}, second ed., Mathematical
  Surveys and Monographs, vol.~69, American Mathematical Society, Providence,
  RI, 1999, Edited and with a foreword by A. A. Ranicki. \MR{1687388
  (2000a:57089)}

\bibitem[Wei13]{Weibel}
C.~A. Weibel, \emph{The {$K$}-{B}ook}, Graduate Studies in Mathematics, vol.
  145, American Mathematical Society, Providence, RI, 2013, An introduction to
  algebraic $K$-theory. \MR{3076731}

\bibitem[Wil05]{Williams2}
B.~Williams, \emph{Quadratic {$K$}-theory and geometric topology}, Handbook of
  {$K$}-theory. {V}ol. 1, 2, Springer, Berlin, 2005, pp.~611--651. \MR{2181831
  (2006k:19011)}

\bibitem[WO93]{Wegge-Olsen}
N.~E. Wegge-Olsen, \emph{{$K$}-theory and {$C^*$}-algebras}, Oxford Science
  Publications, The Clarendon Press, Oxford University Press, New York, 1993, A
  friendly approach. \MR{1222415 (95c:46116)}

\bibitem[WW98]{WW-Duality}
M.~Weiss and B.~Williams, \emph{Duality in {W}aldhausen categories}, Forum
  Math. \textbf{10} (1998), no.~5, 533--603. \MR{1644309 (99g:19002)}

\end{thebibliography}

\end{document}